\documentclass[11pt]{article} 

\usepackage[latin1]{inputenc}
\usepackage[swedish, english]{babel}
\usepackage{epsfig}
\usepackage{float}
\usepackage{subfigure}
\usepackage{latexsym}
\usepackage{amssymb, amsmath, amsthm}
\usepackage{multirow}

\usepackage[normalem]{ulem} % strike through using \sout

\newtheorem{theorem}{Theorem}[section]
\newtheorem*{theorem*}{Theorem}
\newtheorem{corollary}[theorem]{Corollary}
\newtheorem*{corollary*}{Corollary}
\newtheorem{lemma}[theorem]{Lemma} 
\newtheorem*{lemma*}{Lemma} 
 
\newtheorem*{proposition*}{Proposition} 
\newtheorem{assumption}[theorem]{Assumption} 
\newtheorem*{assumption*}{Assumption} 
\newtheorem{remark}[theorem]{Remark} 
\theoremstyle{definition} 
 
\theoremstyle{definition} 
\newtheorem*{definition*}{Definition} 

\title{Inverses of SBP-SAT finite difference operators approximating the first and second derivative}

\begin{document}
\thispagestyle{plain}

\author{Sofia Eriksson\footnote{Department of Mathematics, 
Linnaeus University, 
V\"{a}xj\"{o}, Sweden. Email: sofia.eriksson@lnu.se}}

\date{}

\maketitle

%%%%%%%%%%%%%%%%%%%%%%%%%%%%%%%%%%%%%%%%%%%%%%%%%%%%%%%%%%%%%%%%%%%%%%%%%%%%%%%%%%%%%%%%%%%%%%%%%%%%%%%%%%%%%%%%%%%%%%%%%%%%%%%%%%%%%%%%%%%%%%%%%%%%%%%%%%%%%%%%%%%%%%%%%%%%%%%%%%%%

% DEF WIDEBAR
\newcount\tmpnum \newdimen\tmpdim
{\lccode`\?=`\p \lccode`\!=`\t\lowercase{\gdef\ignorept#1?!{#1}}}
\edef\widecharS{\expandafter\ignorept\the\fontdimen1\textfont1}

\def\widebar#1{\futurelet\next\widebarA#1\widebarA}
\def\widebarA#1\widebarA{%
 \def\tmp{0}\ifcat\noexpand\next A\def\tmp{1}\fi
 \widebarE
 \ifdim\tmp pt=0pt \overline{#1}%
 \else {\mathpalette\widebarB{#1}}\fi
}

\def\widebarB#1#2{%
 \setbox0=\hbox{$#1\overline{#2}$}%
 \tmpdim=\tmp\ht0 \advance\tmpdim by-.4pt
 \tmpdim=\widecharS\tmpdim
 \kern\tmpdim\overline{\kern-\tmpdim#2}%
}
\def\widebarC#1#2 {\ifx#1\end \else 
 \ifx#1\next\def\tmp{#2}\widebarD 
 \else\expandafter\expandafter\expandafter\widebarC
 \fi\fi
}
\def\widebarD#1\end. {\fi\fi}
\def\widebarE{\widebarC A1.4 J1.2 L.6 O.8 T.5 U.7 V.3 W.1 Y.2 
 a.5 b.2 d1.1 h.5 i.5 k.5 l.3 m.4 n.4 o.6 p.4 r.5 t.4 v.7 w.7 x.8 y.8
 \alpha1 \beta1 \gamma.6 \delta.8 \epsilon.8 \varepsilon.8 \zeta.6 \eta.4
 \theta.8 \vartheta.8 \iota.5 \kappa.8 \lambda.5 \mu1 \nu.5 \xi.7 \pi.6
 \varpi.9 \rho1 \varrho1 \sigma.7 \varsigma.7 \tau.6 \upsilon.7 \phi1
 \varphi.6 \chi.7 \psi1 \omega.5 \cal1 \end. }

\def\test#1{$\let\.=#1 \.M, \.A, \.g, \.\beta, \.{\cal A}^q, \.{AB}^\sigma, 
\.H^C, \.{\sin z}, \.W_{\!n}$}

%%%%%%%%%%%%%%%%%%%%%%%%%%%%%%%%%%%%%%%%%%%%%%%%%%%%%%%%%%%%%%%

% independent variables
\newcommand{\dx}{\,\mathrm {d} x}
\newcommand{\ddt}{\frac{\mathrm {d} }{\mathrm {d} t}}
\newcommand{\trans}{\mathsf{T}}%{T}%

\newcommand{\indexl}{\text{\tiny L}}
\newcommand{\indexi}{\text{\tiny I}}
\newcommand{\indexr}{\text{\tiny R}}
\newcommand{\indexlr}{\text{\tiny L,R}}
\newcommand{\indexrl}{\text{\tiny R,L}}

\newcommand{\indexc}{\text{\tiny C}}
\newcommand{\indexd}{\text{\tiny D}}
\newcommand{\indext}{\text{\tiny T}}

% continuous variables:
\newcommand{\domsize}{\ell}
\newcommand{\scu}{u}
\newcommand{\force}{f}%{\mathcal{F}}%
\newcommand{\initial}{u_0}%{\scu_0}

\newcommand{\gl}{g_\indexl}
\newcommand{\gr}{g_\indexr }
\newcommand{\glr}{g_\indexlr }

% discrete variables
\newcommand{\ngpts}{n}
\newcommand{\scv}{\mathbf{v}}
\newcommand{\fh}{\mathbf{\force}}

\newcommand{\scvc}{v}

\newcommand{\el}{\mathbf{e}_\indexl}
\newcommand{\er}{\mathbf{e}_\indexr }
\newcommand{\elr}{\mathbf{e}_\indexlr}
\newcommand{\erl}{\mathbf{e}_\indexrl}

\newcommand{\dsel}{\mathbf{d}_\indexl}
\newcommand{\dser}{\mathbf{d}_\indexr }
\newcommand{\dselr}{\mathbf{d}_\indexlr}
\newcommand{\dserl}{\mathbf{d}_\indexrl}

\newcommand{\PH}{H}

\newcommand{\lindisc}{L}
\newcommand{\genOPsnok}{K}
\newcommand{\GreenDiscGen}{G}

% second derivative
\newcommand{\al}{\alpha_\indexl}
\newcommand{\ar}{\alpha_\indexr}
\newcommand{\bl}{\beta_\indexl}
\newcommand{\br}{\beta_\indexr}
\newcommand{\alr}{\alpha_\indexlr}
\newcommand{\blr}{\beta_\indexlr}

\newcommand{\penI}{\sigma}
\newcommand{\penS}{\tau}

% compact, more or less
\newcommand{\taul}{\penI_\indexl}
\newcommand{\taur}{\penI_\indexr}
\newcommand{\taulr}{\penI_\indexlr}
\newcommand{\sigmal}{\penS_\indexl}
\newcommand{\sigmar}{\penS_\indexr}
\newcommand{\sigmalr}{\penS_\indexlr}

 \newcommand{\dual}{\delta}
 \newcommand{\duall}{\dual_\indexl}
 \newcommand{\dualr}{\dual_\indexr}
 \newcommand{\duallr}{\dual_\indexlr}

\newcommand{\ett}{\mathbf{1}}
\newcommand{\xfet}{\mathbf{x}}
\newcommand{\one}{\vec{1}}
\newcommand{\xvec}{\vec{x}}

 \newcommand{\noll}{\mathbf{0}}
\newcommand{\Imid}{\bar{I}}

\newcommand{\pickup}{{\color{oklar}\mathbf{b}}}
\newcommand{\pickupl}{\mathbf{b}_\indexl}
\newcommand{\pickupr}{\mathbf{b}_\indexr}
\newcommand{\pickuplr}{\mathbf{b}_\indexlr}

\newcommand{\Abeg}{a_\indexl}
\newcommand{\Aend}{a_\indexr }
\newcommand{\Aoff}{a_\indexc}
\newcommand{\Amid}{\bar{A}}
\newcommand{\Awest}{\vec{a}_\indexl}
\newcommand{\Aeast}{\vec{a}_\indexr }
\newcommand{\Avecs}{\vec{a}_\indexlr}

\newcommand{\qhattot}{\xi_\indext}
\newcommand{\qhatc}{\xi_\indexc}
\newcommand{\qhatl}{\xi_\indexl}
\newcommand{\qhatr}{\xi_\indexr}
\newcommand{\qhatlr}{\xi_\indexlr}
\newcommand{\qhatd}{\xi_\indexd}
\newcommand{\qhatall}{\xi_{\indexlr,\indexc}}

% sing
\newcommand{\godtycklig}{\zeta}
 \newcommand{\detmatris}{\Sigma}
 
 \newcommand{\godtyckligare}{\varepsilon}
 \newcommand{\tmpvar}{\Gamma}

\newcommand{\app}{\gamma}
\newcommand{\Atilde}{\tilde{A}_\gamma}

\newcommand{\ql}{q_\indexl}
\newcommand{\qr}{q_\indexr}
\newcommand{\qlr}{q_\indexlr}
\newcommand{\qc}{q_\indexc}
\newcommand{\qtot}{q_\indext}
\newcommand{\qall}{q_{\indexlr,\indexc}}

\newcommand{\qsnokl}{\widetilde{q}_\indexl}
\newcommand{\qsnokr}{\widetilde{q}_\indexr}
\newcommand{\qsnoklr}{\widetilde{q}_\indexlr}
\newcommand{\qsnokc}{\widetilde{q}_\indexc}
\newcommand{\qsnoktot}{\widetilde{q}_\indext}
\newcommand{\qsnokall}{\widetilde{q}_{\indexlr,\indexc}}

\newcommand{\helpvar}{\mathbf{w}}
\newcommand{\Msnok}{\widetilde{M}}
\newcommand{\pert}{p}%{\delta}
 \newcommand{\arbscall}{\rho_\indexl}
 \newcommand{\arbscalr}{\rho_\indexr}
 \newcommand{\arbscallr}{\rho_\indexlr}

\newcommand{\mainl}{s_\indexl}
\newcommand{\mainr}{s_\indexr}
\newcommand{\mainlr}{s_\indexlr}

\newcommand{\minorl}{t_\indexl}
\newcommand{\minorr}{t_\indexr}
\newcommand{\minorlr}{t_\indexlr}

\newcommand{\rl}{r_\indexl}
\newcommand{\rr}{r_\indexr}
\newcommand{\rlr}{r_\indexlr}

\newcommand{\intvar}{y}

% (4,2) narrow
\newcommand{\Opinv}{\vec{g}}
\newcommand{\opinv}{g}
\newcommand{\theR}{\psi}
\newcommand{\Gmid}{\bar{G}}

\newcommand{\evec}{\vec{e}}
 \newcommand{\coeffs}{c}

 \newcommand{\cc}[1]{\coeffs_{1}}
\newcommand{\ck}[1]{\coeffs_{2}}
\newcommand{\ct}[1]{\coeffs_{3}}
 \newcommand{\cf}[1]{ \coeffs_{4}}

 \newcommand{\ucc}[1]{ \coeffs^u_{1}}
\newcommand{\uck}[1]{ \coeffs^u_{2}}
\newcommand{\uct}[1]{ \coeffs^u_{3}}
 \newcommand{\ucf}[1]{ \coeffs^u_{4}}
 
 \newcommand{\lcc}[1]{ \coeffs^l_{1}}
\newcommand{\lck}[1]{ \coeffs^l_{2}}
\newcommand{\lct}[1]{ \coeffs^l_{3}}
 \newcommand{\lcf}[1]{ \coeffs^l_{4}}
 
 \newcommand{\call}{ \coeffs_{1,2,3,4}}

 \newcommand{\callu}{{\coeffs^u_{1,2,3,4}}}
 \newcommand{\calll}{{\coeffs^l_{1,2,3,4}}}

 \newcommand{\nydetskalning}{\mathcal{Q}}
 \newcommand{\nyserie}{\mathcal{P}}
 \newcommand{\corrgreen}{\kappa}

%%%%%%%%%%%%%%%%%%%%%%%%%%%%%%%%%%%%%%%%%%%%%%%%%%%%%%%%%%%%%%%%%%%%%%%%%%%%%%%%%%%%%%%%%%%%%%%%%%%%%%%%%%%%%%%%%%%%%%%%%%%%%%

\newcommand{\Fsnok}{\widetilde{\mathbf \force}}

\newcommand{\Qsnok}{\widetilde{Q}}
\newcommand{\OPsnok}{\widetilde{A}}

% First derivative

\newcommand{\qvec}{\vec{q}}
\newcommand{\fundis}{G_1}
\newcommand{\cfet}{\mathbf{b}}
\newcommand{\Qbar}{\widebar{Q}}
\newcommand{\Ibar}{\widebar{I}}

% Second derivative

\newcommand{\GreenDisc}{G_2}

\newcommand{\preco}{P}%M

\newcommand{\Qinvmid}{\widebar{G}}
\newcommand{\Qinvvec}{\vec{g}}
\newcommand{\Qinvele}{g}
\newcommand{\rtre}{\phi}

 \newcommand{\nyttplus}{\mathcal{D}}%{\mathcal{D}}
 \newcommand{\modif}{\mathcal{B}}%{\mathcal{S}}
 
 \newcommand{\rotursym}{\mathcal{A}}%{\mathcal{X}}
 \newcommand{\storserien}{\mathcal{C}}%{\mathcal{C}}
 \newcommand{\grund}{\nu}

\newcommand{\zl}{z_\indexl}
\newcommand{\zr}{z_\indexr}
\newcommand{\zlr}{z_\indexlr}
\newcommand{\zc}{z_\indexc}

\newcommand{\xl}{x_\indexl}
\newcommand{\xr}{x_\indexr}
\newcommand{\xlr}{x_\indexlr}

\newcommand{\yl}{y_\indexl}
\newcommand{\yr}{y_\indexr}
\newcommand{\ylr}{y_\indexlr}

\begin{abstract}

The scalar, one-dimensional advection equation and heat equation are considered. These equations are discretized in space, using a finite difference method satisfying summation-by-parts (SBP) properties. To impose the boundary conditions, we use a penalty method called simultaneous approximation term (SAT). Together, this gives rise to two semi-discrete schemes where the discretization matrices approximate the first and the second derivative operators, respectively. The discretization matrices depend on free parameters from the SAT treatment. 
 
We derive the inverses of the discretization matrices, interpreting them as discrete Green's functions. In this direct way, we also find out precisely which choices of SAT parameters that make the discretization matrices singular. %, something that is generally difficult to know. 
In the second derivative case, it is shown that if the penalty parameters are chosen such that the semi-discrete scheme is dual consistent, the discretization matrix can become singular even when the scheme is energy stable.

The inverse formulas hold for SBP-SAT operators of arbitrary order of accuracy. For second and fourth order accurate %finite difference 
operators, the inverses are provided explicitly.

\end{abstract}

\section{Introduction}

Consider the time-dependent partial differential equation \eqref{pdegen} below, where $\mathcal{L}$ represents a linear differential operator and $\force(x)$ is a forcing function. We assume that some suitable initial condition and %-- for simplicity homogeneous -- 
-- for the moment homogeneous -- boundary conditions are given such that we have a well-posed problem. Applying the method of lines, that is discretizing first in space while keeping time continuous, yields a system of ordinary differential equations \eqref{discgen}, where we refer to $\lindisc$ as the {\it discretizarion matrix}.
\begin{subequations}
\begin{align}
\label{pdegen}
\scu_t+\mathcal{L}\scu&=\force,&&t\geq0,\quad x\in[0,\domsize],\\
\label{discgen}
\scv_t+\lindisc\scv&=\fh,&& t\geq0.
\end{align}
\end{subequations}
We first look at the scalar advection equation and thereafter at the heat equation, both in one spatial dimension. Thus $\lindisc$ approximates either the first or the second derivative operator, including boundary treatments.

In this paper, $\lindisc$ is obtained using the SBP-SAT finite difference method. This
 class of finite difference %(FD) 
method %(FDM) 
is based on difference operators fulfilling
 summation-by-parts (SBP) properties, and is modified by the penalty technique simultaneous approximation term (SAT) for treating %imposing 
 the boundary conditions. The SBP operators were first developed for first derivatives \cite{ref:KREI74,ref:STRA94} and then later for second derivatives \cite{Carpenter1999341,{Mattsson2004503}} and are designed to facilitate the derivation of energy estimates. A means to impose boundary conditions without destroying these properties is to use SAT \cite{ref:CARP94}. 
The SATs included in $\lindisc$ contain free parameters. 
 We follow the common practice of determining these parameters using the energy method, such that \eqref{discgen} is guaranteed to be time-stable. 
 Thereafter, any remaining degrees of freedom in the SATs can be used to make the scheme {\it dual consistent}. 
 Dual consistency is advantageous when computing functionals of the solution, since the order of accuracy of functionals from dual consistent schemes can be higher compared to those from non-dual consistent schemes \cite{HickenNo17}. 
For more details about SBP-SAT, see \cite{Magnus201417,DelReyFernandez2014171}.

Thanks to the SBP-SAT properties, the discretization matrix can be factorized as $\lindisc=\PH^{-1}\genOPsnok$, where $\PH$ is a symmetric, positive definite matrix that has the role of a quadrature rule, see \cite{Hicken2013111}. Now consider the steady version of \eqref{pdegen}, $\mathcal{L}\scu=\force$. Its solution $ \scu(x)$ may be represented as in \eqref{pdegreen} below, where $\mathcal{G}$ is the Green's function. The steady version of \eqref{discgen} is $\lindisc\scv=\fh$. Solving for $\scv$, yields 
\eqref{discgreen}.
\begin{subequations}
 \begin{align}\label{pdegreen}
 \scu(x)&=\int_0^{ \domsize}\mathcal{G}(x,\intvar)\force(\intvar)\,\mathrm {d} \intvar,
\\\label{discgreen}
\scv&=\genOPsnok^{-1}\PH\fh.
 \end{align}
\end{subequations}
With 
$\PH$'s role as a quadrature rule in mind,
we can see a clear similarity between \eqref{pdegreen} and \eqref{discgreen}, and
realize that
$\genOPsnok^{-1}$ %then
resembles
the Green's function. It makes sense to refer to $\genOPsnok^{-1}$ as a {\it discrete Green's function}.

A finite difference analogue of the Green's function was introduced already %as early as 
in the fundamental article \cite{Courant1928}. 
Thereafter, discrete Green's functions % in various forms 
appear sporadically %now and then
 in the
 literature, see for example \cite{DeeterSpringer1965,CHUNG2000191} and references therein. 
E.g. in \cite{Beyn1982} (and correspondingly in \cite{Courant1928} for two-dimensional problems) the finite formula approximating \eqref{pdegreen} is scaled with the spatial mesh size $h$, which then corresponds closely to \eqref{discgreen}.
 However, since traditional %classical 
 finite difference stencils usually do not have an assigned quadrature rule in the same sense 
 as the SBP operators, the term "discrete Green's functions" often %sometimes 
 refers to $\lindisc^{-1}$ rather than to $\genOPsnok^{-1}$, for example in \cite{CHUNG2000191,Stetter1968,Branden2007}.

In the above-mentioned articles, the standard way of enforcing boundary conditions, %called 
{\it injection}, has been used instead of SAT (for descriptions of these two boundary methods, 
see for example \cite{Magnus201417}).
In \cite{Eriksson20092659}, the first and second derivatives were approximated using %the 
an 
% second order accurate (at most) 
SBP-SAT finite volume method, 
 the inverses analogous to $\genOPsnok^{-1}$
 were derived and
 used for analysing errors. Here, 
we 
derive %general 
formulas for 
$\genOPsnok^{-1}$ %where $\lindisc=\PH^{-1}\genOPsnok$ 
corresponding to
 the first and second derivatives % operator 
 as well, however,
as an extension to the results in \cite{Eriksson20092659}, our formulas hold for
 arbitrary orders of accuracy and in the second derivative case 
%they are extended to hold for
we consider
 general Robin boundary conditions instead of only Dirichlet boundary conditions.

The inverses are full matrices and are therefore %as such maybe %perhaps %they might not be 
probably not %directly 
competitive for solving % finding solutions to linear systems of equations 
systems $\lindisc\scv=\fh$ directly,
compared to fast solvers %of banded, 
for
banded matrices. It is however often advisable to %transform the original problem using a preconditioning matrix 
use pre-conditioning 
to improve the convergence of iterative methods, % for solving systems, 
\cite{GroteHuckle1997}. A 
preconditioning matrix $\preco$ should ideally % be chosen such that $\preco\lindisc\approx I$ 
approximate the inverse of $\lindisc$
in some sense,
and 
knowledge about the %the known 
structure of the inverses could --
speculatively -- be used %for pre-conditioning of iterative solvers. % using pre-conditioning. 
when designing preconditioning matrices.
If $\preco$ is
 a sparse approximate inverse, % (e.g. diagonal, block-diagonal or narrow block-banded), 
 the computations are cheap, but
preconditioners $\preco$ may also be essentially dense matrices, %. Preconditioners of this type are 
as for example the fundamental solution preconditioners considered in \cite{Branden2007}.

The paper is organized as follows: In Section~\ref{FirstDeriv}, we look at the semi-discrete scheme approximating the advection equation. The matrix $\genOPsnok$ associated with $\frac{\partial}{\partial x}$ is denoted $\Qsnok$, and its inverse is presented in Theorem~\ref{ThmGenInv1}. In Section~\ref{SecondDeriv}, we consider the heat equation, 
 thus approximating $\frac{\partial^2}{\partial x^2}$.
The related matrix $\genOPsnok$, denoted $\OPsnok$, is inverted in Theorem~\ref{ThmGenInv}. 
The SAT parameters are chosen to give stability and dual consistency, and %in addition to that 
additionally
it is of interest to know if some choices of SAT parameters %may 
result in a singular discretization matrix $\lindisc$.
In the second derivative case, it turns out that an energy stable scheme can actually have a singular $\lindisc$ if the scheme is also dual consistent. Some relations between stability, dual consistency and a singular discretization matrix is discussed in Section~\ref{RelStabSingDual}. We also discuss the relations between two different ways of showing energy stability, in Section~\ref{relationsErikssonDual}. The paper is summarized in Section~\ref{Summary}.

\section{The first derivative}
\label{FirstDeriv}

Consider the scalar advection equation with a Dirichlet boundary condition at the inflow boundary, that is 
\begin{align}
\label{contPDE1}
\begin{array}{rll}
\scu_t+
\scu_{x}=&\hspace{-7pt}\force,\hspace{20pt}&x\in[0,\domsize],\hspace{20pt}\vspace{4pt}\\ \scu=&\hspace{-7pt}\gl,&x= 0,
\end{array}
\end{align}
valid for $t\geq0$, with initial condition $\scu(x,0)=\initial(x)$. The forcing function $\force(x,t)$, the initial data $\initial(x)$ and the boundary data $\gl(t)$ are %supposed to be 
known functions.

%\subsection{Well-posedness}

We call \eqref{contPDE1} well-posed if it has a unique solution and is stable (can be bounded by data). 
Techniques for showing existence and uniqueness can be found in for example \cite{KL,GKO}. We focus on showing
stability, since we will derive a corresponding stable discrete problem later.
We use the energy method, and multiply the 
partial differential equation in \eqref{contPDE1} by $\scu$, and integrate over the spatial domain. Thereafter, we use integration by parts and apply the boundary condition.
For simplicity, we consider the homogeneous case, that is with the data $\force=0$ and $\gl=0$. % (this is actually to show uniqueness).
This yields
\begin{align*}
\ddt\|\scu\|^2%=2\int_0^\domsize\force\scu\dx-\scu(\domsize,t)^2+\scu(0,t)^2
%=2\int_0^\domsize\scu\force\dx-\scu(\domsize,t)^2+\gl^2
=-\scu(\domsize,t)^2
\end{align*}
where $\|\scu\|^2=\int_0^{\domsize}\scu^2\dx$ and where we have used that $(\scu^2)_t=2\scu\scu_t$.
In the homogeneous case, the growth rate thus becomes $\ddt\|\scu\|^2\leq0$. Integrating this in time yields %we obtain 
the energy
estimate $\|\scu\|^2\leq\|\initial\|^2$ and the %problem is well-posed. 
solution is thus bounded. 
Since \eqref{contPDE1} is an one-dimensional hyperbolic
 problem it is also possible to show strong well-posedness, i.e., that $\|\scu\|$ is bounded by the data $\force$, $\gl$ and $\initial$. See \cite{KL,GKO} for different definitions of well-posedness.

\subsection{The semi-discrete scheme}
\label{Semi1}

We first discretize in space, on the interval $x\in[0, \domsize]$, using $\ngpts +1$ equidistant grid points $x_i=ih$, where $h= \domsize/\ngpts $ and $i=0,1,\hdots,\ngpts $.
Using the SBP-SAT finite difference method, %see \cite{Magnus201417,DelReyFernandez2014171}, 
we obtain a semi-discrete scheme approximating \eqref{contPDE1} as
\begin{align}
\label{discSCHEME1}
\begin{split}
\scv_t+D_1\scv=\fh &+\PH ^{-1} \taul \el\left( \el^\trans\scv-\gl\right),
\end{split}
\end{align}
where $\scv(t)=[\scvc_0, \scvc_1, \hdots, \scvc_\ngpts ]^\trans$ is the approximation of the continuous solution $\scu(x,t)$, and
where 
$\fh=[\force(x_0,t), \force(x_1,t), \hdots, \force(x_\ngpts ,t)]^\trans$ is the restriction of $\force(x,t) $ to the grid. 
In the same way, we let the initial data be %$\scv(0)=\initv=[\initial(x_0), \initial(x_1), \hdots, \initial(x_N)]^\trans$. 
$\scv(0)=[\initial(x_0), \initial(x_1), \hdots, \initial(x_\ngpts )]^\trans$. 
The matrix $D_1$ approximates the first derivative operator $\partial/\partial x$, 
and fulfills the SBP-properties \cite{ref:KREI74,ref:STRA94}
\begin{align}\label{SBPprop1}
D_1=\PH ^{-1}Q,&&\PH=\PH^\trans>0,&&Q+Q^\trans=\er\er^\trans-\el\el^\trans
\end{align}
where 
$\el=[1, 0, \hdots, 0]^\trans$ and $\er=[ 0, \hdots, 0, 1]^\trans$.
By the notation $>$, we mean that the matrix $\PH$ is %a %symmetric, 
positive definite. % matrix.
% Note that 
As mentioned in the introduction,
 $\PH$ has the role of a quadrature rule and %that it 
 $\|\scv\|_{\PH}^2\equiv\scv^\trans\PH\scv$
approximates the $L^2$-norm of $\scu(x,t)$, see \cite{Hicken2013111}. 
The scalar $\taul$ determines the strength of the SAT, %and %should be chosen properly 
and will be chosen below 
such that the scheme \eqref{discSCHEME1} is energy stable and dual consistent.

\subsubsection{Stability and dual consistency}

To show energy stability, we multiply \eqref{discSCHEME1} by $\scv^\trans\PH$ from the left and use the relations \eqref{SBPprop1}.
We thereafter add the transpose, 
and % for simplicity
 we consider $\fh=0$ and $\gl=0$, just as in the continous case. This yields 
\begin{align*}
\ddt\|\scv\|^2_\PH=-\scvc_\ngpts ^2+(1+2 \taul) \scvc_0^2,
\end{align*}
where
$\scvc_0=\el^\trans\scv$ and $\scvc_\ngpts =\er^\trans\scv$. %For zero data, w
We need $\ddt\|\scv\|^2_\PH\leq0$, which is guaranteed if $ \taul\leq-1/2$.
For a dual consistent scheme, we need $\taul=-1$, see \cite{HickenNo17,Berg20126846}.

\subsection{The inverse of the discretization matrix}

We first rewrite \eqref{discSCHEME1} as
\begin{align}\label{SchemeShort1}
\scv_t+
\PH^{-1}\Qsnok\scv&=\Fsnok,
\end{align}
where 
\begin{align}\label{Qsnok}
\Qsnok=Q- \taul \el \el^\trans,&&\Fsnok=\fh -\PH ^{-1} \taul \el\gl.
\end{align}
We identify $\Qsnok$ as the first derivative version of $\genOPsnok$ discussed in the introduction.
The second order accurate version of $\Qsnok$ was inverted in \cite{Eriksson20092659} and inspired by those results, we make a similar ansatz and derive $\Qsnok^{-1}$ of arbitrary order of accuracy. The result is given in Theorem~\ref{ThmGenInv1}.

\begin{theorem}\label{ThmGenInv1}

%Consider the $(\ngpts+1) \times (\ngpts+1) $-matrix $\Qsnok=Q- \taul \el \el^\trans$ found in \eqref{Qsnok}, with $Q$ from \eqref{SBPprop1}. 
Consider the $(\ngpts+1) \times (\ngpts+1) $-matrices $Q$ from \eqref{SBPprop1} and $\Qsnok$ found in \eqref{Qsnok}.
The structures of $Q$ and $\Qsnok$ are
\begin{align}\label{Qparts}
Q=\left[\begin{array}{cc}-1/2&\qvec^\trans\\-\qvec&\Qbar\end{array}\right],&&
\Qsnok=\left[\begin{array}{cc}-1/2-\taul&\qvec^\trans\\-\qvec&\Qbar\end{array}\right],
\end{align}
where $\qvec$ is an $\ngpts \times1$-vector and $\Qbar$ is an $\ngpts \times \ngpts $-matrix.
The inverse of $\Qsnok$ 
is 
\begin{align}\label{QsnokInv}
%\Qsnok^{-1}=\fundis-\frac{1}{\taul}\randis,
\Qsnok^{-1}=\fundis-\frac{1}{\taul}\ett\cfet^\trans,
\end{align}
 where
 \begin{align}\label{GreenEtc}
%\fundis=\left[\begin{array}{cc}0&0\\0&\Qbar^{-1}\end{array}\right],&&\randis=\ett\cfet^\trans,
\fundis=\left[\begin{array}{cc}0&0\\0&\Qbar^{-1}\end{array}\right],&&\ett=[1,1,\hdots,1]^\trans,&&\cfet^\trans=\left[\begin{array}{cc}1&-\qvec^\trans\Qbar^{-1}\end{array}\right].
\end{align}
%and where $\ett=[1,1,\hdots,1]^\trans$ and $\cfet^\trans=\left[\begin{array}{cc}1&-\qvec^\trans\Qbar^{-1}\end{array}\right]$.

\end{theorem}

\begin{proof}[Proof of Theorem~\ref{ThmGenInv1}]
We aim to show that $\Qsnok\Qsnok^{-1}=I$, where $I$ is the $(\ngpts +1)\times (\ngpts +1)$ identity matrix. 
Using $\Qsnok$ from \eqref{Qsnok} and $\Qsnok^{-1}$ from \eqref{QsnokInv}, 
we compute
\begin{align*}
\Qsnok\Qsnok^{-1}&=\left(Q-\taul \el\el^\trans\right)\left(\fundis-\frac{1}{\taul}\ett\cfet^\trans\right)\\
&=Q\fundis-\frac{1}{\taul}Q\ett\cfet^\trans-\taul\el\el^\trans\fundis+\el\el^\trans\ett\cfet^\trans.
%=\left[\begin{array}{cc}0&\qvec^\trans\Qbar^{-1}\\0&\Ibar\end{array}\right]+\left[\begin{array}{cc}1&-\qvec^\trans\Qbar^{-1}\\0&0\end{array}\right]=I
\end{align*}
%First, n
Note that $D_1\ett=0$, since $D_1$ in \eqref{SBPprop1} is a consistent difference operator. Hence, %Consequently, 
$Q\ett=0$.
Furthermore, $\el^\trans\fundis=0$ since the first row of $\fundis$ consists of zeros. These relations, the fact that $\el^\trans\ett=1$ and the structures of the components in \eqref{Qparts} and \eqref{GreenEtc} yields
\begin{align*}
\Qsnok\Qsnok^{-1}
&=Q\fundis+\el\cfet^\trans
=\left[\begin{array}{cc}0&\qvec^\trans\Qbar^{-1}\\0&\Ibar\end{array}\right]+\left[\begin{array}{cc}1&-\qvec^\trans\Qbar^{-1}\\0&0\end{array}\right]=I
\end{align*}
where $\Ibar$ is the $\ngpts \times \ngpts $ identity matrix.
 \end{proof}

 \begin{corollary} %$\Qsnok$ is non-singular if $\taul\neq0$.
 \label{Cor1}
The structure of $\Qsnok^{-1}$ in \eqref{QsnokInv} implies that $\Qsnok$ is %non-singular if $\taul\neq0$.
singular if $\taul=0$.
\end{corollary}

The existence of $\fundis$ and $\cfet$ in \eqref{GreenEtc}, and consequently 
the validity of Theorem~\ref{ThmGenInv1}, %and Corollary~\ref{Cor1},
%both
 rely on the %existence of $\Qbar^{-1}$. 
 assumption that $\Qbar$ is invertible. 
In the (2,1) order accurate case --
where we by the notation "(2,1) order accurate",
 refer to a matrix $D_1$ which has second order of accuracy in the interior finite difference stencil and first order of accuracy at the boundaries --
 the inverse of
 $\Qbar$ %explicit expression of %$\Qbar^{-1}$ 
is derived and presented %explicitly
 in Section~\ref{InversesExamples1}, which directly proves its existence. 
The same is done for the inverse of the (4,2) order accurate operator, which is presented in Section~\ref{InversesExamples14}.

Higher order operators, on the other hand, have free parameters.
For example, % in the 
% case, %for example, 
 for the diagonal norm (6,3) order accurate version of $D_1$ described in \cite{ref:STRA94}, $x_1$ is a free parameter.
In this case, we find numerically that $\Qbar$ is singular %(or very, very close to being singular) 
when $x_1\approx0.69$.

\begin{remark}
\label{RemGreen1}

For the steady version of \eqref{contPDE1}, that is $\scu_{x}=\force$ with $\scu(0)=\gl$,
we have
\begin{align*}
\scu(x)&=\gl+\int_0^{ \domsize}\mathcal{G}(x,\intvar)\force(\intvar)\,\mathrm {d} \intvar,
&
\mathcal{G}(x,\intvar)=\left\{\begin{array}{ll}1,&\intvar<x,\\0,&x\leq\intvar,\end{array}\right.
\end{align*}
where $\mathcal{G}$ is a 
Green's function.
Starting from $\scv=\Qsnok^{-1}\PH\Fsnok$, using \eqref{Qsnok} and \eqref{QsnokInv} as well as the relations $\cfet^\trans \el=1$ and $\fundis\el=\noll$ deduced from \eqref{GreenEtc}, we obtain
\begin{align*}
\scv%&=\Qsnok^{-1}(\PH\fh -\taul \el\gl)=(\fundis-\frac{1}{\taul}\ett\cfet^\trans)(\PH\fh -\taul \el\gl)\\
%&=\fundis\PH\fh -\taul \fundis\el\gl-\frac{1}{\taul}\ett\cfet^\trans\PH\fh +\ett\cfet^\trans \el\gl\\
%&=\left(\fundis -\frac{1}{\taul}\ett\cfet^\trans\right)\PH\fh +\ett\gl\\
&=\gl\ett+\Qsnok^{-1}\PH\fh .
\end{align*}
%Recalling that the matrix $\PH$ has the role of a quadrature rule, we see the clear similarity to the continuous case, i.e. the inverse 
Recall from the introduction that
$\genOPsnok^{-1}=\Qsnok^{-1}$ resembles
% a
%Green's function %$\mathcal{G}(x,\intvar)$
$\mathcal{G}$. 
The version of $\Qsnok^{-1}$ found in \eqref{Qinv2} in Appendix~\ref{InversesExamples1} (which corresponds to the second order accurate operator) is
\begin{align*}
\left(\Qsnok^{-1}\right)_{i,j}=\left\{\begin{array}{ll}
1-(1+1/\taul)(-1)^j,&0\leq j\leq i\leq \ngpts ,\\
(-1)^{i+j}-(1+1/\taul)(-1)^j,& 0\leq i \leq j\leq \ngpts .
\end{array}\right.
\end{align*}
%With %employing 
% Injection corresponds to having $\taul=-\infty$, and thus results in $\Qsnok^{-1}=\fundis$.
The dual consistent choice $\taul=-1$ %the second order accurate version of $
is optimal in the sense that it
cancels %minimizes
 the oscillations
such that % when $j\leq i$ and gives 
$(\Qsnok^{-1})_{i,j}=1%=\mathcal{G}(x_i,x_j)
$ for $j\leq i$, however %although %but %. However,
 $(\Qsnok^{-1})_{i,j}=(-1)^{i+j}\neq0$ %\mathcal{G}(x_i,x_j)$ 
 for $i\leq j$. 
If we instead let $\taul\to-\infty$, interpreted as mimicking the injection treatment, results in $\Qsnok^{-1}=\fundis$.
\end{remark}

\section{The second derivative}
\label{SecondDeriv}

Now consider the scalar heat equation with Robin boundary conditions, that is 
\begin{align}
\label{contPDE}
\begin{array}{rll}
\scu_t-
\scu_{xx}=&\hspace{-7pt}\force ,\hspace{20pt}&x\in[0, \domsize],\hspace{20pt}\vspace{4pt}\\ \al\scu-\bl\scu_{x}=&\hspace{-7pt}\gl,&x= 0,\vspace{4pt}\\\ar\scu+\br\scu_x=&\hspace{-7pt}\gr,&x= \domsize,
\end{array}
\end{align}
valid for $t\geq0$, with initial condition $\scu(x,0)=\initial(x)$. The forcing function $\force(x,t)$, the initial data $\initial(x)$ and the boundary data $\glr(t)$ are %compatible smooth functions.
known functions.

%We use the energy method again, m
We multiply the %partial differential equation (PDE)
partial differential equation in \eqref{contPDE} by $\scu$ and integrate the result over the spatial domain, %we consider the homogeneous case, i.e. 
with the data put to $\force=0$ and $\glr=0$. %, for simplicity.
Thereafter using integration by parts and the boundary conditions, yields
%\begin{align}\label{energyinnanbc}
%\ddt\|\scu\|^2+
%2\|\scu_x\|^2=2\int_0^1\scu\force\dx -2\gl\scu_{x}(0,t)+2\gr\scu_{x}(1,t),
%\end{align}
\begin{align*}
%\int_0^1\scu\scu_t\dx+\int_0^1\scu_{x}^2\dx=& [\scu\scu_{x}]_0^1 \\
%\ddt\|\scu\|^2+2\|\scu_x\|^2=&2 \scu(1,t)\scu_{x}(1,t)-2 \scu(0,t)\scu_{x}(0,t)\\
\ddt\|\scu\|^2+2\|\scu_x\|^2=&-2 \frac{\br}{\ar}\scu_{x}( \domsize,t)^2-2 \frac{\bl}{\al}\scu_{x}(0,t)^2.
\end{align*}
%where $\|\scu\|^2=\int_0^{ 1}\scu^2\dx$ and where we have used that $(\scu^2)_t=2\scu\scu_t$.
For %well-posedness, 
a decaying growth rate, we need $\alr\blr\geq0$.

\subsection{The semi-discrete scheme}

%We first discretize in space, on the interval $x\in[0,1]$, using $N+1$ equidistant grid points $x_i=ih$, where $h=1/N$ and $i=0,1,\hdots,N$.
Using the SBP-SAT finite difference method, %see \cite{Magnus201417,DelReyFernandez2014171}, 
we obtain a %semi-discrete 
scheme approximating \eqref{contPDE} as
\begin{align}
\label{discSCHEME}
\begin{split}
%\scv_t-D_2\scv=\fh &+\PH ^{-1}( \taul \el -\sigmal S^\trans \el) \left( \al\el^\trans\scv-\bl\el^\trans S\scv-\gl\right)\\&+\PH ^{-1}(\taur \er +\sigmar S^\trans \er) \left( \ar\er^\trans\scv+\br\er^\trans S\scv-\gr\right),
\scv_t-D_2\scv=\fh &+\PH ^{-1}( \taul \el -\sigmal \dsel) \left( \al\el^\trans\scv-\bl\dsel^\trans \scv-\gl\right)\\&+\PH ^{-1}(\taur \er +\sigmar \dser) \left( \ar\er^\trans\scv+\br\dser^\trans \scv-\gr\right),
\end{split}
\end{align}
where $\scv$, $\fh$, $\PH$ and $\elr$ are described as in Section~\ref{Semi1}.
The matrix $D_2$ approximates the second derivative operator, 
and fulfills the SBP-properties
\begin{align}\label{SBPprop2}
%D_2=\PH ^{-1}(-A+(E_N-E_0)S),&&\hspace{23pt}A=A^\trans=S^\trans MS\geq0,
D_2=\PH ^{-1}(-A+\er\dser^\trans-\el\dsel^\trans),&&\hspace{23pt}A=A^\trans\geq0.
\end{align}
%where $\el=[1, 0, \hdots, 0]^\trans$ and $\er=[ 0, \hdots, 0, 1]^\trans$.
The vectors $\dsel$ and $\dser$ are consistent finite difference stencils approximating the first derivative, see \cite{Carpenter1999341}. 
Two common categories of $D_2$ operators are {\it wide-stencil} and {\it narrow-stencil} operators. Wide-stencil operators can be factorized as $D_2= D_1^2$, and the term "narrow" describes finite difference schemes with a minimal stencil width \cite{MattComp}.

The penalty parameters $\taulr$ and $\sigmalr$ in \eqref{discSCHEME} are scalars that will be further specified and discussed in the next sections. Now, we use \eqref{SBPprop2} to rewrite \eqref{discSCHEME} as
\begin{align}\label{SchemeShort}
\scv_t+
\PH^{-1}\OPsnok\scv&=\Fsnok,
\end{align}
where %%$\OPsnok= -\PH D_2-( \taul \el -\sigmal \dsel)(\al \el^\trans-\bl \dsel^\trans ) -(\taur \er +\sigmar \dser) ( \ar \er^\trans+\br\dser^\trans)$ or
\begin{align}\label{Kop2}
\begin{split}\OPsnok
%&= -\PH D_2-( \taul \el -\sigmal \dsel)(\al \el^\trans-\bl \dsel^\trans ) -(\taur \er +\sigmar \dser) ( \ar \er^\trans+\br\dser^\trans)\\
&=\hspace{-1pt}A-\left[\hspace{-5pt}\begin{array}{c}\el^\trans\\-\dsel^\trans\end{array}\hspace{-4pt}\right]^\trans\left[\hspace{-3pt}\begin{array}{cc}\taul\al&1+\taul\bl\\\sigmal\al&\sigmal\bl\end{array}\hspace{-3pt}\right]\left[\hspace{-5pt}\begin{array}{c}\el^\trans\\-\dsel^\trans\end{array}\hspace{-4pt}\right]
-\left[\hspace{-3pt}\begin{array}{c}\er^\trans\\\dser^\trans\end{array}\hspace{-4pt}\right]^\trans\left[\hspace{-3pt}\begin{array}{cc}\taur\ar&1+\taur\br\\\sigmar\ar&\sigmar\br\end{array}\hspace{-3pt}\right]\left[\hspace{-3pt}\begin{array}{c}\er^\trans\\\dser^\trans\end{array}\hspace{-4pt}\right]
\end{split}\end{align}
and where $\Fsnok=\fh -\PH ^{-1}( \taul \el -\sigmal \dsel)\gl-\PH ^{-1}(\taur \er +\sigmar \dser) \gr$.
We identify $\OPsnok$ as the second derivative version of the matrix
$\genOPsnok$ from % mentioned in
 the introduction.

\subsubsection{Stability}
\label{stability2}

To show energy stability, we multiply \eqref{discSCHEME} by $\scv^\trans\PH$ from the left and use the relations \eqref{SBPprop2}.
We thereafter add the transpose, 
and %for simplicity we
 let $\fh=0$ and $\glr=0$. %, just as in the continous case. 
 This yields
\begin{align}\label{BeforeStabtricks}\begin{split}
\ddt\|\scv\|^2_\PH+2\scv^\trans A\scv&=2\scv^\trans(\er\dser^\trans -\el\dsel^\trans )\scv\\
&+2\scv^\trans( \taul \el -\sigmal \dsel) \left( \al\el^\trans\scv-\bl\dsel^\trans \scv\right)\\
&+2\scv^\trans(\taur \er +\sigmar \dser)\left( \ar\er^\trans\scv+\br\dser^\trans \scv\right),\end{split}
\end{align}
where we %, for zero data, 
need to show that $\ddt\|\scv\|^2_\PH\leq0$.
We will determine the stability limits of $\taulr$ and $\sigmalr$
using a procedure sometimes called the {\it borrowing technique} %using the sometimes so-called
\cite{Carpenter1999341,ref:GONG06,APPELO2007531,MATTSSON20088753,SVARD20084805,WangKreiss2017,ORDER-PRESERVING}. The idea is
 to "borrow" a maximum amount $\app$ of "positivity" %in form of $( \dsel\dsel^\trans+ \dser\dser^\trans)$ 
 from $A$, more precisely as %according to %such that a remaining $\Atilde\geq0$ still holds
\begin{align}\label{BorrowProp}
%A=\Atilde+h\app S^\trans (E_N-E_0)^\trans(E_N-E_0)S,&&\Atilde\geq0,\quad\app>0.
A=\Atilde+h\app (\dsel\dsel^\trans+\dser\dser^\trans),&&\Atilde\geq0,\quad\app>0.
\end{align}
Inserting the relation in \eqref{BorrowProp} into %Inserting this into
 \eqref{BeforeStabtricks}, we obtain
\begin{align*}
\ddt\|\scv\|^2_\PH+2\scv^\trans \Atilde\scv&=\left[\begin{array}{c}\el^\trans\scv\\-\dsel^\trans\scv \end{array}\right]^\trans\left[\begin{array}{cc} 2\taul \al&1+\taul\bl +\sigmal \al\\1+\taul\bl+\sigmal \al&2\sigmal\bl-2h\app\end{array}\right] \left[\begin{array}{c}\el^\trans\scv\\-\dsel^\trans\scv \end{array}\right]\\&+\left[\begin{array}{c}\er^\trans\scv\\\dser^\trans \scv\end{array}\right]^\trans\left[\begin{array}{cc} 2\taur \ar&1+\taur\br+\sigmar \ar \\1+\taur\br+\sigmar \ar &2\sigmar\br-2h\app\end{array}\right] \left[\begin{array}{c}\er^\trans\scv\\\dser^\trans\scv \end{array}\right].
\end{align*}
For stability, we need both the matrices in the two quadratic forms above to be negative semi-definite. 
% \begin{align*}
%\begin{split}
%\left[\begin{array}{cc} 2\taul \al&-\sigmal \al-\taul\bl-1 \\-\sigmal \al-\taul\bl-1&2+\sigmal\bl-\frac{2}{\q}\end{array}\right] \leq0
%\end{split}
%\end{align*}
This is fulfilled if
\begin{align}\label{stabbeforedual}\begin{split}
2\taulr \alr&\leq0\\2(\sigmalr\blr-h\app)&\leq0\\
(1+\sigmalr \alr+\taulr\blr)^2&\leq4\taulr \alr(\sigmalr\blr-h\app).\end{split}
\end{align}
%that is
%\begin{align}\label{borrowstab2}\begin{split}
%\taulr \alr\leq0\hspace{30pt}\sigmalr\blr\leq h\app\\
%(1-\sigmalr \alr+\taulr\blr)^2\leq-4 \alr(\taulr h\app+\sigmalr )
%\end{split}
%\end{align}
%or

\subsubsection{Dual consistency}

To make the scheme \eqref{discSCHEME} dual consistent 
% \begin{align*}
%\taulr\blr+1=\sigmalr\alr.
% \end{align*}
we first note that the operator %problem \eqref{contPDE}
 $\partial^2/\partial x^2$ (including boundary conditions) 
 is a symmetric operator and that the matrix $\OPsnok$ must be symmetric to mimic this.
From %\eqref{Kop}
\eqref{Kop2}
 it is clear that %$\taulr\blr+1=\sigmalr\alr$
$\OPsnok$ is symmetric if $1+\taulr\blr=\sigmalr\alr$.
%
%Inspired by this, w
Let
\begin{align}\label{delta}
\duall\equiv1 +\taul\bl-\sigmal\al&&\dualr\equiv1+ \taur\br- \sigmar\ar,
\end{align}
 where
% $\delta_L=1-\sigmal\al +\taul\bl$ and $\delta_R=1- \sigmar\ar+ \taur\br$
$\duallr=0$
% are zero
 for dual consistent choices of penalty parameters. 
The relations in \eqref{delta}, with $\duallr=0$, can also be derived from the penalty parameters of the scalar problem in \cite{ErikssonDual}.
For a background and more thorough descriptions  %more details about
of dual consistency, see \cite{HickenNo17}.

%These three requirements can be reformulated as
Note that now, 
using the dual consistency parameters $\duallr$ defined in \eqref{delta},
the three stability requirements in \eqref{stabbeforedual} can be reformulated as
\begin{align}\label{stabexpdual}
\taulr \alr\leq0,&&\sigmalr\blr\leq h\app,&&
\duallr^2\leq-4 \alr(\taulr h\app+\sigmalr ). %,
\end{align}
%where $\duallr$ are the dual consistency measures found in \eqref{delta}.

\subsection{The inverse of the discretization matrix}

We consider the steady version of \eqref{SchemeShort}, that is $\PH^{-1}\OPsnok\scv=\Fsnok$, %where we can obtain the solution $\scv$ if we know $\OPsnok^{-1}$. 
which has a unique solution $\scv=\OPsnok^{-1}\PH\Fsnok$, if  $\OPsnok^{-1}$ exists.
%the inverse of $\OPsnok$.
We derive this inverse %in Appendix~\ref{ProofThmInv} 
and present the result 
%The inverse is given
in Theorem~\ref{ThmGenInv}.

\begin{theorem}\label{ThmGenInv}

Consider $\OPsnok$ in \eqref{Kop2}, which depends on $A$ and $\dselr$ in \eqref{SBPprop2} %and on the penalty parameters $\taulr$ and $\sigmalr$. 
and on the boundary related scalars $\taulr$, $\sigmalr$, $\alr$ and $\blr$. % in \eqref{discSCHEME}.
Let the parts of $A$ 
be denoted as follows,
\begin{align}\label{Aparts}
A=\left[\begin{array}{ccc}\Abeg&\Awest^\trans&\Aoff\\\Awest&\Amid&\Aeast\\\Aoff&\Aeast^\trans&\Aend\end{array}\right],
\end{align}
where 
$\Abeg$, $\Aend$ and $\Aoff$
 are scalars, 
 $\Avecs$
 are $(\ngpts -1)\times1$-vectors and $\Amid$ is an $(\ngpts -1)\times(\ngpts -1)$-matrix. 
The inverse of $\OPsnok$ 
is 
 \begin{align}\label{OPsnokinvRobin}
\OPsnok^{-1}&=\GreenDisc+\left[\begin{array}{cccc}-\sigmal\pickupl& -\sigmar\pickupr&\ett-\xfet/\domsize & \xfet/\domsize \end{array}\right] \detmatris^{-1} \left[\begin{array}{c}\pickupl^\trans\\ \pickupr^\trans\\\bl(\ett-\xfet/\domsize )^\trans\\ \br\xfet^\trans /\domsize\end{array}\right]
 \end{align}
 where 
 $\ett=[1\ 1\ 1\ \hdots\ 1]^\trans$ and $\xfet=h[0\ 1\ 2\ \hdots\ \ngpts]^\trans$, and where
 \begin{align}\label{InverseParts}
\GreenDisc=\left[\begin{array}{ccc}0&0&0\\0&\Amid^{-1}&0\\0&0&0\end{array}\right],&&
\pickupl\equiv\ett-\xfet/\domsize -\GreenDisc\dsel,&&\pickupr\equiv\xfet /\domsize+\GreenDisc\dser.
\end{align} 
Furthermore, $ \detmatris$ in \eqref{OPsnokinvRobin} is a $4\times4$-matrix
 \begin{align}\label{determining}
 \detmatris=\left[\begin{array}{cccc}\taul+\sigmal\qhatl&-\sigmar\qhatc& 0&0\\ -\sigmal\qhatc&\taur+\sigmar\qhatr&0&0\\\duall&0&\al+\bl/\domsize &-\bl/\domsize \\ 0&\dualr&- \br/\domsize & \ar+\br/\domsize \end{array}\right]
 \end{align}
that depends on $\alr$ and $\blr$, that is on the choices of boundary conditions in \eqref{contPDE}, on the choices of penalty parameters $\taulr$ and $\sigmalr$ in \eqref{discSCHEME} and on the duality parameters $\duallr$ in \eqref{delta}, as well as on the scalars
 \begin{align}\label{qsnokdefALT}
\qhatl \equiv- \dsel^\trans \pickupl,&&\qhatr \equiv\dser^\trans\pickupr&& \qhatc \equiv\dsel^\trans\pickupr=-\dser^\trans\pickupl.
 \end{align}
\end{theorem}

\begin{proof}[Proof of Theorem~\ref{ThmGenInv}]
%We aim to show that $\OPsnok\OPsnok^{-1}=I$, where $\OPsnok$ is given in \eqref{Kop2}, the derived $\OPsnok^{-1}$ is given in \eqref{OPsnokinvRobin} and $I$ is the identity matrix. This is done in Appendix~\ref{ProofThmInv}.
The proof is given in Appendix~\ref{ProofThmInv}.
\end{proof}

Note that the quantities in \eqref{InverseParts}, and thus the validity of Theorem~\ref{ThmGenInv},
rely on the existence of $\Amid^{-1}$. %As mentioned above, for the operators of order two and four,
%For some operators of lower order, the explicit expression of $\Amid^{-1}$ is given 
%
In Appendix~\ref{InversesExamples}, %we provide 
the explicit values of %the components 
$\Amid^{-1}$, as well as of
$\GreenDisc$, $\pickuplr$, $\qhatlr$ and $\qhatc$, are provided
for the (2,0), (2,1) and (4,2) order accurate narrow-stencil operators and the (2,0) order accurate wide-stencil operator. 
%If we want to avoid the detour of using $\pickuplr$, we can also compute
%
%For the operators considered in Appendix~\ref{InversesExamples}, $\Amid^{-1}$ is given explicitly which directly proves its existence.
This directly proves the existence of $\Amid^{-1}$ for %the lower order accurate operators considered in Appendix~\ref{InversesExamples}.
these operators.
Higher order accurate operators have free parameters, 
%but empirically (numerically) it can be confirmed that the inverse $\Amid^{-1}$ exists for the parameter choices used in \cite{Mattsson2004503}.
but empirically we can draw the  conclusion that $\Amid^{-1}$ must exist  at least %for example 
for the parameter choices in \cite{Mattsson2004503}, 
%has been confirmed empirically (by extensive successful use) that $\Amid^{-1}$ exists.
 since  the operators therein have been applied successfully for many years.

Given the existence of $\Amid^{-1}$, we note that $\OPsnok$ in \eqref{OPsnokinvRobin} is singular if and only if %the determinant of the $4\times4$-matrix 
$ \detmatris$ in \eqref{determining} % is zero, that is if 
is singular.
%Note that %(given the existence of $\Amid^{-1}$ mentioned in Remark~\ref{Abarinverterbar}) 
%$\OPsnok$
The matrix $ \detmatris$ is in turn singular if any of the two relations
%\begin{itemize}
%\item[1.] $(\al+\bl/\domsize )( \ar+\br /\domsize)-\bl\br/\domsize^2=0$
%\item[2.] $(\taul+\sigmal\qhatl)(\taur+\sigmar\qhatr)-\sigmal\sigmar\qhatc^2=0$
%\end{itemize}
\begin{align}
(\al+\bl/\domsize )( \ar+\br /\domsize)-\bl\br/\domsize^2=0\label{singcond1}\\\label{singcond2}
(\taul+\sigmal\qhatl)(\taur+\sigmar\qhatr)-\sigmal\sigmar\qhatc^2=0
\end{align}
holds. The first condition %has to do with
is related to
 the continuous boundary conditions, and makes the matrix %$\OPsnok$
 singular if Neumann boundary conditions are imposed on both boundaries, i.e. if $\al=\ar=0$.
 The second condition has to do with the choice of penalty parameters, and leads us to the following corollary of Theorem~\ref{ThmGenInv}:

\begin{corollary}\label{CorGenInv}

The matrix $\OPsnok$, described in \eqref{Kop2}, is singular when the penalty parameters simultaneous fulfill 
$\taul =-\left(\qhatl+\godtycklig |\qhatc|\right)\sigmal$ and $\taur =-\left(\qhatr+|\qhatc|/\godtycklig\right)\sigmar$,
where $\godtycklig\neq0$. 
If $\qhatc$, $\sigmal$ or $\sigmar$ is zero, the matrix $\OPsnok$ is singular if either $\taul =-\sigmal\qhatl$ or if $\taur =-\sigmar\qhatr$.
\end{corollary}

\begin{proof}[Proof of Corollary~\ref{CorGenInv}]
We make the ansatz $\taulr =-\sigmalr\qhatlr-\godtyckligare_\indexlr$ with some unknown scalars $\godtyckligare_\indexlr$. Inserting this into \eqref{singcond2} above gives $\godtyckligare_\indexl\godtyckligare_\indexr=\sigmal\sigmar\qhatc^2$ which is fulfilled for all pairs $\godtyckligare_\indexl=\sigmal|\qhatc|\godtycklig$ and $\godtyckligare_\indexr=\sigmar|\qhatc|/\godtycklig$ with arbitrary choices of $\godtycklig\neq0$.
If $\qhatc$, $\sigmal$ or $\sigmar$ is equal to zero, it is enough if either $\godtyckligare_\indexl=0$ or $\godtyckligare_\indexr=0$.
\end{proof}

The requirements on $A$ and $\dselr$ in
Theorem~\ref{ThmGenInv} are only that $A$ is symmetric, that $\Amid^{-1}$ exists (as discussed above) and that $D_2$ and $\dselr$ in \eqref{SBPprop2} are consistent such that
the relations
\eqref{Sconsistent} and
\eqref{AxA1-x}
in Appendix~\ref{ProofThmInv} holds.
In addition we will
assume that $D_2$ is constructed such the left and right boundary closures are equivalent. This implies that $A$ is a centrosymmetric matrix, that is $A_{i,j} = A_{\ngpts-i,\ngpts-j}$ for all $0 \leq i,j \leq \ngpts$, and that $(\dsel)_i = -(\dser)_{\ngpts-i}$ for $0 \leq i \leq \ngpts$.
This additional assumption leads to $\qhatl=\qhatr$ (this is easiest seen by expressing the quantities in \eqref{qsnokdefALT} as
% \begin{align*}
%\qhatl \equiv 1/\domsize+\dsel^\trans\GreenDisc\dsel,&&\qhatr \equiv1/\domsize+\dser^\trans\GreenDisc\dser,&& \qhatc \equiv1/\domsize+\dsel^\trans\GreenDisc\dser=1/\domsize+\dser^\trans\GreenDisc\dsel.
% \end{align*}}% 
 $\qhatlr = 1/\domsize+\dselr^\trans\GreenDisc\dselr$ and $\qhatc =1/\domsize+\dselr^\trans\GreenDisc\dserl$ and thereafter using the fact that the inverse of a centrosymmetric matrix is also centrosymmetric).
For later reference we define
\begin{align}
\label{qhat}
\qhattot\equiv\qhatlr+|\qhatc|,
\end{align}
and assume  that the penalty is chosen to be equally strong on both boundaries: % to simplify Corollary~\ref{CorGenInv}:

\begin{assumption}%\begin{remark}
\label{remmuqu}
%Choosing $\godtycklig=1$ makes the condition for singularity in Corollary~\ref{CorGenInv} identical for the two boundaries, yielding $\taulr =-\q\sigmalr$. 
Choosing %Assuming 
an equal %boundary treatment 
penalty strength
on both boundaries corresponds to having $\godtycklig=1$ in Corollary~\ref{CorGenInv}.
If in addition 
%assuming 
equivalent boundary closures
are assumed, such that $\qhatl=\qhatr$, we can use
% assuming that $A$ is constructed according to Assumption~\ref{AssCentro} such that $\qhatl=\qhat_N$, 
 %$\qhat$
$\qhattot\equiv\qhatlr+|\qhatc|$ from \eqref{qhat}.
%Taken together, t
This simplifies the condition of singularity in Corollary~\ref{CorGenInv} to $\taulr =-\qhattot\sigmalr$.
\end{assumption}%\end{remark}

\begin{remark}
\label{RemGreen2}
%The interior of $\OPsnok^{-1}$, that is $\GreenDisc$ in \eqref{InverseParts}, 
The inverse of $\OPsnok$ 
mimics %resembles
 a fundamental solution.
For example, the
Green's function 
$\mathcal{G}$ of 
Poisson's equation, $-\scu_{xx}=\force$ with $\scu(0)=\scu( \domsize)=0$,
 is
\begin{align*}
\scu(x)&=\int_0^{ \domsize}\mathcal{G}(x,\intvar)\force(\intvar)\,\mathrm {d} \intvar,%=(1-x)\int_0^x\intvar \force(\intvar)d\intvar+x\int_x^1(1-\intvar) \force(\intvar)d\intvar
&
%G(x,\intvar)=\left\{\begin{array}{ll}x(1-\intvar)&x\leq\intvar\\
%\intvar(1-x)&x>\intvar\end{array}\right.
\mathcal{G}(x,\intvar)=\left\{\begin{array}{ll}\intvar(1-x/\domsize ),&\intvar<x,\\x(1-\intvar/\domsize ),&x\leq\intvar.\end{array}\right.
\end{align*}
Recalling that the matrix $\PH$ has the role of a quadrature rule, %and %that 
%it approximates 
%approximating
%the %$L^2$ norm 
%inner product \cite{Hicken2013111}),
we see the clear similarity to the %steady version of \eqref{SchemeShort}, $\scv=\OPsnok^{-1}\PH\Fsnok$. 
time-independent, homogeneous version of \eqref{SchemeShort}, $\scv=\OPsnok^{-1}\PH\fh$. % ($\Fsnok=\fh$ since $\glr=0$). 
%Recall that the matrix $\PH$ has the role of a quadrature rule, %and %that 
%%it approximates 
%approximating
%the %$L^2$ norm 
%inner product \cite{Hicken2013111}.
The resemblance is more obvious if 
 %ignoring %the part stemming from the boundary treatment in \eqref{OPsnokinv},
 the penalty dependent part in \eqref{OPsnokinvRobin} is ignored,
 since then %$\scv=\GreenDisc\PH\Fsnok$. 
$\scv=\GreenDisc\PH\fh$.
For the second order accurate approximation, $\GreenDisc$ is exact in the grid points, see \eqref{Ksnokinv2int}.
This is identical with the result noted for the classical finite difference method using {\it injection} instead of SAT, compare
\cite{Beyn1982,Stetter1968}.
 \end{remark}

\subsection{Relations between stability, singularity and dual consistency }
\label{RelStabSingDual}

We take a look at the relation between the stability requirements on the scheme \eqref{discSCHEME} and the conditions that make its discretization matrix singular.
First, we note that:
\begin{theorem}\label{happ=1/qsnok} Consider $\app$ in \eqref{BorrowProp} and $\qhattot$ in \eqref{qhat}.
It holds that
$h\app=1/\qhattot$.
\end{theorem}
\begin{proof}Theorem~\ref{happ=1/qsnok} is proven in Appendix~\ref{happqsnok=1}.
\end{proof}
 A consequence of Theorem~\ref{happ=1/qsnok} is that the stability demands in \eqref{stabexpdual} can be written
\begin{align}\label{stabexpdualNEW}
\taulr \alr\leq0,&&\sigmalr\blr\leq 1/\qhattot,&&
\duallr^2\leq-4 \alr(\taulr/\qhattot +\sigmalr ),
\end{align}
with $\duallr$ from \eqref{delta}.
We will see that the penalty can be chosen such that we have energy stability and a singular discretization matrix at the same time:
From Assumption~\ref{remmuqu} we know that
the matrix $\OPsnok$ is singular when %the penalty parameters simultaneous fulfill $\taul =-\qhat\sigmal$ and $\taur =-\qhat\sigmar$.
$\taulr=-\sigmalr\qhattot$.
% $\taulr +\q\sigmalr=0$.
%\begin{align}\label{singNEW}
%\taul =-\qhat\sigmal,&&
%\taur =-\qhat\sigmar.
%\end{align}
Inserting this into 
\eqref{stabexpdualNEW}, the third stability demand becomes
%If $\taulr=-\sigmalr\qhat$ the matrix $\OPsnok$ singular. 
%The stability demands \eqref{stabexpdualNEW} now become
%\begin{align*}
%-\sigmalr \alr\leq0,&&\q\sigmalr\blr\leq 1,&&
%\duallr^2\leq0,
%\end{align*}
%where the third requirement 
%Moreover, the third stability demand in \eqref{stabexpdualNEW} becomes
$\duallr^2\leq0$,
which is only fulfilled if the penalty parameters are chosen in a dual consistent way. This means that if \eqref{discSCHEME} is an energy stable scheme, it must also be dual consistent to risk having a singular discretization matrix.
 Note though that even if the scheme is dual consistent, a singular discretization matrix is avoided by choosing $\taulr\neq-\sigmalr\qhattot$.
To be precise, simultaneous having
%\begin{align*}\taulr=\frac{-\qhattot}{\blr\qhattot+\alr },&&\sigmalr=\frac{1 }{\blr\qhattot +\alr}\end{align*}
$\taulr=-\qhattot/(\blr\qhattot+\alr )$ and $\sigmalr=1 /(\blr\qhattot +\alr)$ should be avoided, since 
this particular %peculiar %cumbersome 
choice makes $\duallr=0$, 
fulfills the stability demands but at the same time
 makes
$\OPsnok$ singular.

%Strictly speaking, 
In Assumption~\ref{remmuqu},
one can argue that $\godtycklig=-1$ gives just as an equal 
penalty strength as $\godtycklig=1$, simplifying Corollary~\ref{CorGenInv} to
%$\taul =-\left(\qhatl- |\qhatc|\right)\sigmal$ and $\taur =-\left(\qhatr-|\qhatc|\right)\sigmar$. 
$\taulr =-\left(\qhatlr-|\qhatc|\right)\sigmalr$. 
However, these choices do not give energy stability and %hence they 
are therefore not  interesting for our further discussions. %, since %we will focus on the role of $\qhattot$. 
Besides,  $|\qhatc|$ tend to be very small so in practice it does not make much of a difference.
%\eqref{stabexpdualNEW}
%\begin{align*}
%%\taulr \alr\leq0,&&\sigmalr\blr\leq 1/\qhattot,&&\duallr^2\leq-8 \sigmalr \alr(|\qhatc| )/\qhattot,\\
%\taulr \alr\leq0,&&\sigmalr\blr\leq 1/\qhattot,&&
%\duallr^2\leq8 \taulr\alr|\qhatc|  /(\qhattot\left(\qhatlr-|\qhatc|\right))\leq0,\\
%\taulr \alr\leq0,&&\sigmalr\blr\leq 1/\qhattot,&&
%1 +\taul\bl-\sigmal\al=0,\\
%%\Rightarrow\sigmal=\frac{1}{\al+\left(\qhatlr-|\qhatc|\right)\bl} &%&\sigmal\bl=1/\qhattot+\frac{2|\qhatc|\bl}{\qhattot(\al+\left(\qhatlr-|\qhatc|\right)\bl)}-\frac{\al}{\qhattot(\al+\left(\qhatlr-|\qhatc|\right)\bl)}
%\Rightarrow\alr=0&&1 +\taul\bl=0&&\sigmalr\bl=\frac{1}{\left(\qhatlr-|\qhatc|\right)}=\frac{1}{\left(\qhatlr-|\qhatc|\right)}\geq1/\qhattot
%\end{align*}

\subsection{Relations to the stability demands in \cite{ErikssonDual}}
\label{relationsErikssonDual}

In Section~\ref{stability2} the "borrowing technique" 
is used for deriving the stability restrictions on the penalty parameters. %method.
In \cite{ErikssonDual}, a different approach (inspired by \cite{HickenNo17,Berg20126846} where wide-stencil discretizations are rewritten as first order systems) is used for showing stability, and here we are going to comment on some connections between the two methods.

In \cite{ErikssonDual}, it is assumed that $A$ can be decomposed as in \cite{Carpenter1999341}, that is as
 \begin{align}\label{decompAS}
A=A^\trans=S^\trans MS,&&\dsel=S^\trans \el,&&\dser=S^\trans \er,
\end{align}
and the strategy for showing stability %(inspired by \cite{HickenNo17,Berg20126846} where wide-stencil discretizations are rewritten as first order systems) 
is 
 to modify the approximation of $\scu_x$ from $S\scv$ to the auxiliary variable $\helpvar=S\scv+M^{-1}\el\arbscall +M^{-1}\er\arbscalr$. In \cite{ErikssonDual}, $\arbscallr$ are penalty-like terms proportional to the solution deviations from boundary data, but other options are possible.
Computing $\helpvar^\trans M\helpvar$
makes the terms 
\begin{align*}
%\helpvar^\trans M\helpvar-\scv^\trans A\scv=
2\scv^\trans \dsel\arbscall +2\scv^\trans \dser\arbscalr +\ql\arbscall^2 +2\qc \arbscall \arbscalr+\qr\arbscalr^2 \leq2\scv^\trans (\dsel\arbscall+\dser\arbscalr) +\qtot(\arbscall^2 +\arbscalr^2 )
\end{align*} 
available to the boundary terms in \eqref{BeforeStabtricks}, 
where $\qlr$, $\qc$ and $\qtot$ %(the latter under Assumption~\ref{AssCentro})
are defined 
as
\begin{align}\label{qdefcompact}
%\q\equiv \el^\trans M^{-1}\el+|\el^\trans M^{-1}\er|=\er^\trans M^{-1}\er+|\el^\trans M^{-1}\er|.
 \qlr\equiv\elr^\trans M^{-1}\elr,&&\qc \equiv\el^\trans M^{-1}\er=\er^\trans M^{-1}\el, &&\qtot\equiv\qlr+|\qc |.
%\q_{0}\equiv\el^\trans M^{-1}\el, &&\q_{\ngpts}\equiv\er^\trans M^{-1}\er, &&\qc \equiv\el^\trans M^{-1}\er=\er^\trans M^{-1}\el.
\end{align}
The "borrowing technique" on the other hand,
makes
the terms $%\scv^\trans \Atilde\scv-\scv^\trans A\scv=
-h\app\scv^\trans (\dsel\dsel^\trans+\dser\dser^\trans)\scv$ available for the boundary terms in \eqref{BeforeStabtricks}.

Although these two approaches of showing stability are different, they are closely related. 
% It turns out that $\qtot=1/(h\app)$, which we formalize in Lemma~\ref{happ=1/q} below.
In Lemma~\ref{happ=1/q} we formalize this relation and show that $\qtot=1/(h\app)$.
\begin{lemma}\label{happ=1/q}Assume that $A$ in \eqref{SBPprop2} can be factorized as in \eqref{decompAS} with $M>0$, and define $\qtot$ as stated in \eqref{qdefcompact}.
Next, consider \eqref{BorrowProp}, where the parameter $\app$ is defined as the maximum number such that $\Atilde\geq0$ still holds.
Then it holds that $h\app=1/\qtot$.
\end{lemma}
\begin{proof}Lemma~\ref{happ=1/q} is proven in Appendix~\ref{happq=1}.
\end{proof}

For wide-stencil operators, $S=D_1$ and $M=\PH$ in \eqref{decompAS},
and the parameters $\qlr$ and $\qc$ in \eqref{qdefcompact}
 are easily obtained % for the wide-stencil operators, 
 since $M$ %$M=\PH$
 is known.
For narrow-stencil operators on the other hand, $M$ and the interior of $S$ are not uniquely defined.
In \cite{ErikssonDual}, the strategy was (under the contrary %forced
 assumption that $S$ is non-singular and $M$ is singular) to 
 compute 
\begin{align}\label{qsnokdef}
 \qsnoklr\equiv\elr^\trans \Msnok^{-1}\elr, &&\qsnokc\equiv\el^\trans \Msnok^{-1}\er=\er^\trans \Msnok^{-1}\el,&&\qsnoktot\equiv\qsnoklr+|\qsnokc|
%\qsnokl\equiv\el^\trans \Msnok^{-1}\el,&& \qsnokr\equiv\er^\trans \Msnok^{-1}\er, &&\qsnokc\equiv\el^\trans \Msnok^{-1}\er=\er^\trans \Msnok^{-1}\el,
\end{align}
instead,
where 
$\Msnok\equiv S^{-\trans}(A+\pert \el\el^\trans)S^{-1}$
 with $\pert\neq0$ being a perturbation parameter. 
For 
 wide-stencil
 operators though, it can easily be checked numerically that $\qlr\neq\qsnoklr$ and $\qc\neq\qsnokc$.
 This is somewhat alarming, but it can as easily be checked that it still holds that $\qtot=\qsnoktot$.
 We confirm this analytically in Theorem~\ref{GrandFinale} below, and the use of $\qsnoktot$ in \cite{ErikssonDual} is thus justified. First though, we note the following:
\begin{lemma}\label{qsnoksame}%The variables $\qsnoklr$ and $\qsnokc$ defined in \eqref{qsnokdef} are identical to the ones defined in \eqref{qsnokdefALT}. 
The quantities $\qsnoklr$ and $\qsnokc$ defined in \eqref{qsnokdef} are identical to the quantities $\qhatlr$ and $\qhatc$ in \eqref{qsnokdefALT}.
\end{lemma}
 \begin{proof}Lemma~\ref{qsnoksame} is proven in Appendix~\ref{partII}.
\end{proof}

Thus, in summary, 
we have that:
\begin{theorem}
\label{GrandFinale}
Assume that $A$ in \eqref{SBPprop2} can be factorized as in \eqref{decompAS} with $M>0$, and define $\qtot$ as stated in \eqref{qdefcompact}.
Next, assume that $M$ is singular instead, with $M\geq0$, and define $\qsnoktot$ as stated in \eqref{qsnokdef}. Then it holds that $\qtot=\qsnoktot$.
\end{theorem}
\begin{proof}
From Lemma~\ref{happ=1/q} we have that $\qtot=1/(h\app)$ and 
from Theorem~\ref{happ=1/qsnok} we have that %says
$1/(h\app)=\qhattot$.
Combining Lemma~\ref{qsnoksame} with the definitions in \eqref{qsnokdef} and \eqref{qhat} we deduce that 
$\qhattot=\qsnoktot$. All in all, this gives
%\begin{align*}
%%\qlr+|\q_c|=\qsnoklr+|\qsnokc|=1/(h\app)
%\q=1/(h\app)=\qhat=\qsnok
%\end{align*}
$\qtot=1/(h\app)=\qhattot=\qsnoktot$
concluding the proof.
\end{proof}

For an example, see the derived values of $\qsnokall$ and $\qall$ for the wide-stencil (2,0) order operator in Appendix~\ref{AppWide}.
 As a numerical confirmation, in Table~\ref{qcomp} we compare the values of %$h\qtot =h\qsnoktot $ 
 $h\qsnoktot $
 from \cite{ErikssonDual} to the values of $\app$ computed in \cite{MATTSSON20088753,WangKreiss2017}. % for some operators. 
In Table~\ref{qcomp} though, it appears 
that 
 $h\qsnoktot\geq1/\app$. % 
This is because the listed $\app$ % listed in \cite{MATTSSON20088753} 
are computed for $\ngpts\to\infty$, 
and are as such slightly too large
for very coarse meshes.

\begin{table}[h]
\centering
\begin{tabular}{|c|ll|l|}\hline
{Order}&$h\qsnoktot $ from \cite{ErikssonDual}&&$1/\app$ from \cite{MATTSSON20088753,WangKreiss2017}\\\hline
(2,0)&1&&--\\
(2,1)&\underline{2.5}& & \underline{2.5} \\
(4,2)&\underline{3.9863}91480987749 &(for $\ngpts=8$)& \underline{3.9863}50339 \\
(6,3)&\underline{5.322}804652661742 &(for $\ngpts=12$)& \underline{5.322}787044\\
(8,4)&\underline{633.6}9326893357 &(for $\ngpts=16$)& \underline{633.6}2285\\
(10,5)&--&&28.4736205\\\hline
\end{tabular} 
\caption{% The $\q$-values (scaled with $h$) computed in \cite{ErikssonDual}, for narrow-stencil second derivative operators from \cite{Mattsson2004503} and \cite{MATTSSON2013418}. In comparison the borrowing parameter $\app$ from \cite{MATTSSON20088753,WangKreiss2017}. 
The borrowing parameter $\app$ computed in \cite{MATTSSON20088753,WangKreiss2017}, for narrow-stencil second derivative operators from \cite{Mattsson2004503,MATTSSON2013418}. In comparison the $\qsnoktot$-values (scaled with $h$) from \cite{ErikssonDual}.
}
\label{qcomp}
\end{table}

%\section{Conclusions}
\section{Conclusions}
\label{Summary}

We discretize the scalar advection equation and the heat equation in one-dimensional space, using the SBP-SAT finite difference method. This gives rise to two semi-discrete schemes of the form $\scv_t+\lindisc\scv=\Fsnok$, where the discretization matrix $\lindisc$ is approximating either the first derivative or the second derivative, including treatment of the boundary conditions. The matrix $\lindisc$ is, due to properties of the SBP-SAT method, associated with a positive definite matrix $\PH$ such that $\lindisc=\PH^{-1}\genOPsnok$, where the inverse of $\genOPsnok$ is interpreted as a discrete Green's function. 
We derive the general forms of these inverses, and provide explicit examples of $\genOPsnok^{-1}$ for some operators $\lindisc$ of second and fourth order accuracy. 

The boundary treatment SAT induces free parameters in $\lindisc$.
We first determine these parameters such that the semi-discrete schemes are energy stable. Any remaining degrees of freedom can be used to make the schemes dual consistent.
Another important question is whether the discretization matrices $\lindisc$ are invertible.
Conveniently, the formula for $\genOPsnok^{-1}$ reveals precisely which combinations of SAT parameters that make
$\lindisc$ singular. 

% For the first derivative, $\lindisc$ can only be singular if the SAT parameter is equal to zero (at least as long as the SBP operator is properly constructed). This is an expected result since a zero valued SAT means that no boundary condition is imposed.

In the second derivative case, it turns out that for one very particular choice of SAT parameters, $\lindisc$ can become singular even when the scheme is energy stable. Here, we can avoid this and instead choose the parameters such that the scheme is energy stable, dual consistent and guaranteed to have an invertible discretization matrix (and consequently a unique solution). However, for more complex problems it might not be feasible to {\it prove} that the discretization matrix is invertible, not even for energy stable schemes.
%--
%even if it  that is likely the case for energy stable schemes.

Last, we take a look at two supposedly different approaches of proving energy stability. Curiously, they are closely related, leading to the same demands on the SAT parameters.

\subsection*{Acknowledgements}

I would like to thank Jan Nordstr\"{o}m and Anna Nissen for inspiration and encouragement in early discussions about this work.

\appendix

\section{Explicit inverses of the first derivative operator}

\subsection{The (2,1) order accurate operator}
\label{InversesExamples1}

In the second order case, we have
%\begin{align}\label{D1Qsnok}
%%\Dsnok=\frac{1}{2h}\left[\begin{array}{cccccc}-2-4\tau&2\\-1&0&1\\&-1&0&1\\&&\ddots&\ddots&\ddots\\&&&-1&0&1\\&&&&-2&2\end{array}\right]
%D_1=\frac{1}{h}\hspace{-2pt}\left[\hspace{-2pt}\begin{array}{cccccc}-1&1\\-\frac{1}{2}&0&\frac{1}{2}\\&-\frac{1}{2}&0&\frac{1}{2}\\&&\ddots&\ddots&\ddots\\&&&-\frac{1}{2}&0&\frac{1}{2}\\&&&&-1&1\end{array}\hspace{-2pt}\right]\hspace{-3pt},\hspace{4pt}
%%\Qsnok=\frac{1}{2}\left[\begin{array}{cccccc}-1-2\tau&1\\-1&0&1\\&-1&0&1\\&&\ddots&\ddots&\ddots\\&&&-1&0&1\\&&&&-1&1\end{array}\right]
%\Qsnok=\hspace{-3pt}\left[\hspace{-2pt}\begin{array}{cccccc}-\frac{1}{2}-\taul&\frac{1}{2}\\-\frac{1}{2}&0&\frac{1}{2}\\&-\frac{1}{2}&0&\frac{1}{2}\\&&\ddots&\ddots&\ddots\\&&&-\frac{1}{2}&0&\frac{1}{2}\\&&&&-\frac{1}{2}&\frac{1}{2}\end{array}\hspace{-2pt}\right]
%\end{align}
\begin{align}\label{D1Qsnok}\scalebox{.96}{$
D_1=$}\frac{1}{h}\scalebox{.96}{$\left[\begin{array}{cccccc}-1&1\\-\frac{1}{2}&0&\frac{1}{2}\\&-\frac{1}{2}&0&\frac{1}{2}\\&&\ddots&\ddots&\ddots\\&&&-\frac{1}{2}&0&\frac{1}{2}\\&&&&-1&1\end{array}\right],\hspace{4pt}
\Qsnok=\left[\begin{array}{cccccc}-\frac{1}{2}-\taul&\frac{1}{2}\\-\frac{1}{2}&0&\frac{1}{2}\\&-\frac{1}{2}&0&\frac{1}{2}\\&&\ddots&\ddots&\ddots\\&&&-\frac{1}{2}&0&\frac{1}{2}\\&&&&-\frac{1}{2}&\frac{1}{2}\end{array}\right]$}
\end{align}
with the associated norm-matrix $\PH =h\ \text{diag}\left(\begin{array}{ccccccc}\frac{1}{2},& 1,& 1,& \hdots,& 1,& 1,& \frac{1}{2}\end{array}\right)$.
In \eqref{D1Qsnok}, we identify $\qvec^\trans=\left[\begin{array}{ccccc}\frac{1}{2}&0&\hdots&0&0\end{array}\right]$ and $\Qbar$ (given below) according to \eqref{Qparts}.
Using Gauss--Jordan
elimination we find the inverse of
$\Qbar$, as
% leading to
\begin{align*}
%\Qbar=\frac{1}{2}\left[\begin{array}{cccccc}0&1\\-1&0&1\\&\ddots&\ddots&\ddots\\&&-1&0&1\\&&&-1&1\end{array}\right],&&%\qvec=\frac{1}{2}\left[\begin{array}{c}1\\0\\\vdots\\0\\0\end{array}\right],
\Qbar=\frac{1}{2}\left[\begin{array}{ccccccc}0&1\\-1&0&1\\&-1&0&1\\&&\ddots&\ddots&\ddots\\&&&-1&0&1\\&&&&-1&1\end{array}\right]&&
\Longrightarrow&&\Qbar^{-1}=2\left[\begin{array}{cccccc}1&-1&1&-1&1&\cdots\\
1&0&0&0&0&\cdots\\
1&0&1&-1&1&\cdots\\
1&0&1&0&0&\cdots\\
1&0&1&0&1&\cdots\\
\vdots&\vdots&\vdots&\vdots&\vdots&\ddots\end{array}\right].
\end{align*}
We compute $\qvec^\trans\Qbar^{-1}=\left[\begin{array}{cccccc}1&-1&1&-1&1&\cdots\end{array}\right]$ as well.
Inserting these results into \eqref{QsnokInv} and \eqref{GreenEtc} yields 
\begin{align}\label{Qinv2}
%\Qsnok^{-1}=\fundis-\frac{1}{\taul}\randis,
\Qsnok^{-1}=2\left[\begin{array}{ccccccc}0&0&0&0&0&0&\cdots\\
0&1&-1&1&-1&1&\cdots\\
0&1&0&0&0&0&\cdots\\
0&1&0&1&-1&1&\cdots\\
0&1&0&1&0&0&\cdots\\
0&1&0&1&0&1&\cdots\\
\vdots&\vdots&\vdots&\vdots&\vdots&\vdots&\ddots\end{array}\right]-\frac{1}{\taul}\left[\begin{array}{ccccccc}1&-1&1&-1&1&-1&\cdots\\
1&-1&1&-1&1&-1&\cdots\\
1&-1&1&-1&1&-1&\cdots\\
1&-1&1&-1&1&-1&\cdots\\
1&-1&1&-1&1&-1&\cdots\\
1&-1&1&-1&1&-1&\cdots\\
\vdots&\vdots&\vdots&\vdots&\vdots&\vdots&\ddots\end{array}\right],
\end{align}
which we recognize from \cite{Eriksson20092659}.

\subsection{The (4,2) order accurate operator}
\label{InversesExamples14}

In \cite{ref:STRA94}, we find $D_1$ with fourth order interior accuracy and the associated $\PH $.
Together with \eqref{SBPprop1} and \eqref{Qsnok}, this gives us
\begin{align}
\label{Qtildeinv4}
\Qsnok=
\left[\begin{array}{ccccccccccc}%{cccc|ccc|cccc}
-\frac{1}{2}-\taul & \frac{59}{96} & -\frac{1}{12} & -\frac{1}{32} & 0 & 0&0&0&0 &0&0 \\
 -\frac{59}{96} & 0 & \frac{59}{96} & 0 & 0 & 0 &0&0&0&0&0 \\
 \frac{1}{12} & -\frac{59}{96} & 0 & \frac{59}{96} & -\frac{1}{12} & 0 &0&0&0&0&0 \\
 \frac{1}{32} & 0 & -\frac{59}{96} & 0 & \frac{2}{3} & -\frac{1}{12} &0&0&0&0&0 \\%\hline
 0 & 0 & \frac{1}{12} & -\frac{2}{3} & 0 & \frac{2}{3} & -\frac{1}{12} &0&0&0&0\\
 \vdots & \vdots & & \ddots & \ddots & \ddots & \ddots& \ddots&&\vdots&\vdots \\
0 & 0 & 0 & 0 & \frac{1}{12} & -\frac{2}{3} & 0 & \frac{2}{3} & -\frac{1}{12} & 0 & 0\\%\hline
0 & 0 & 0 & 0 & 0 & \frac{1}{12} & -\frac{2}{3} & 0 & \frac{59}{96} & 0 & -\frac{1}{32}\\
0 & 0 & 0 & 0 & 0 & 0 & \frac{1}{12} & -\frac{59}{96} & 0 & \frac{59}{96} & -\frac{1}{12}\\
0 & 0 & 0 & 0 & 0 & 0 & 0 & 0 & -\frac{59}{96} & 0 & \frac{59}{96}\\
0 & 0 & 0 & 0 & 0 & 0 & 0 & \frac{1}{32} & \frac{1}{12} & -\frac{59}{96} & \frac{1}{2}\\
 \end{array}\right].
 \end{align}
 We identify $\Qbar$ and $ \qvec$ as indicated in \eqref{Qparts}.
We are now looking for a matrix $\Qinvmid$ such that
$\Qbar\Qinvmid=\Ibar$. Let $\Qinvmid$ be composed as
\begin{align*}
\Qinvmid
 &=\left[\begin{array}{ccccc} \Qinvvec_{1}&\Qinvvec_{2}&\hdots&\Qinvvec_{\ngpts} \end{array}\right]
 ,&& \Qinvvec_{j}=\left[\begin{array}{ccccc} \Qinvele_{1,j}&\Qinvele_{2,j}&\hdots&\Qinvele_{\ngpts,j} \end{array}\right]^\trans.
\end{align*}
For $\Qbar\Qinvmid=\Ibar$ to hold, 
$\Qbar\Qinvvec_j=\evec_j$ must be fulfilled for all $j=1,2,\hdots,\ngpts$, 
where the $\ngpts\times1$ vector $\evec_j=[0, \hdots, 0, 1, 0, \hdots, 0]^\trans$ is non-zero only in its $j$th element.
For $\Qbar\Qinvvec_j=\evec_j$ to be fulfilled, the interior rows lead to $\Qinvele_{i-2,j}-8\Qinvele_{i-1,j}+8\Qinvele_{i+1,j}- \Qinvele_{i+2,j} =12\delta_{i,j}$,
where $\delta_{i,j}$ is the Kronecker delta. 
Hence, the fourth order linear homogeneous recurrence relation $\Qinvele_{i-2,j} -8\Qinvele_{i-1,j}+8\Qinvele_{i+1,j}-\Qinvele_{i+2,j}=0$ has to be fulfilled by most $\Qinvele_{i,j}$. The general, explicit solution to this recursive relation has the form $\Qinvele_{i,j}=\cc{j}+\ck{j}(-1)^i+\ct{j}\rtre^i+\cf{j}\rtre^{-i}$, where $\rtre=4+\sqrt{15}\approx 7.873$ and where $\call$ are $j$-dependent constants.

The requirement $\Qbar\Qinvvec_j=\evec_j$ takes slightly different forms depending on $j$. %$j=1,2,\hdots,\ngpts$.
For $j=1$, 
%the requirement $\Qbar\Qinvvec_j=\evec_j$, leads to
we have $\Qbar\Qinvvec_1=\evec_1$, which is expressed explicitly as
\begin{align}\label{firstcolumn}\frac{1}{96}
\left[\begin{array}{c}
59\Qinvele_{2,1} \\
 -59\Qinvele_{1,1}+59\Qinvele_{3,1}-8\Qinvele_{4,1} \\
 -59\Qinvele_{2,1}+64\Qinvele_{4,1}-8\Qinvele_{5,1} \\%\hline
 8\left(\Qinvele_{2,1} - 8\Qinvele_{3,1}+8\Qinvele_{5,1}- \Qinvele_{6,1} \right) \\
 \vdots\\
 8\left(\Qinvele_{i-2,1} - 8\Qinvele_{i-1,1}+8\Qinvele_{i+1,1}- \Qinvele_{i+2,1} \right) \\
\vdots\\8\left(\Qinvele_{\ngpts-6,1}-8\Qinvele_{\ngpts-5,1}+8\Qinvele_{\ngpts-3,1}- \Qinvele_{\ngpts-2,1} \right)\\%\hline
8\Qinvele_{\ngpts-5,1} -64\Qinvele_{\ngpts-4,1}+59\Qinvele_{\ngpts-2,1}-3\Qinvele_{\ngpts,1}\\
8\Qinvele_{\ngpts-4,1} -59\Qinvele_{\ngpts-3,1}+ 59\Qinvele_{\ngpts-1,1}-8\Qinvele_{\ngpts,1}\\
 -59\Qinvele_{\ngpts-2,1}+ 59\Qinvele_{\ngpts,1}\\
 3\Qinvele_{\ngpts-3,1} +8\Qinvele_{\ngpts-2,1}- 59\Qinvele_{\ngpts-1,1}+48\Qinvele_{\ngpts,1}
\end{array}\right]=\left[\begin{array}{c}1\\0\\0\\%\hline
0\\\vdots\\0\\\vdots\\0\\%\hline
0\\0\\0\\0\end{array}\right].
\end{align}
The ansatz $\Qinvele_{i,1}=\cc{j}+\ck{j}(-1)^i+\ct{j}\rtre^i+\cf{j}\rtre^{-i}$ holds for $2\leq i\leq \ngpts-2$ where $\call$ are unknowns to be determined. In addition, we have the three unknowns $\Qinvele_{1,1}$, $\Qinvele_{\ngpts-1,1}$ and $\Qinvele_{\ngpts,1}$.
The three first and the four last rows in 
\eqref{firstcolumn} %the system above
 gives us seven conditions. 
Inserting the above mentioned expressions for $\Qinvele_{i,1}$ into \eqref{firstcolumn}, gives a linear system with seven unknowns and seven conditions, as
\begin{align*}
\left[\begin{array}{ccccccc|c}0&59&59&59\rtre^2&59\rtre^{-2} &0&0&96\\ 
-59&51&-67&59\rtre^3-8\rtre^4&59\rtre^{-3}-8\rtre^{-4} &0&0&0\\ 
0&-3&13&-59\rtre^2+8\rtre^{3}&-59\rtre^{-2}+8\rtre^{-3}&0&0 &0\\
0&3&-13(-1)^{\ngpts}&\rtre^{\ngpts}(-8\rtre^{-3}+59\rtre^{-2})&\rtre^{-\ngpts}(-8\rtre^{3}+59\rtre^{2})&0&-3&0\\
0&-51&67(-1)^{\ngpts}&\rtre^{\ngpts}(8\rtre^{-4}-59\rtre^{-3})&\rtre^{-\ngpts}(8\rtre^{4}-59\rtre^{3})& 59&-8&0\\ 
0&-59&-59(-1)^{\ngpts}&-59\rtre^{\ngpts-2}&-59\rtre^{2-\ngpts}&0& 59&0\\ 
0&11&5(-1)^{\ngpts}&\rtre^{\ngpts}(3\rtre^{-3}+8\rtre^{-2})&\rtre^{-\ngpts}(3\rtre^{3}+8\rtre^{2})&- 59&48&0\end{array}\right]
\end{align*}
with the unknowns sorted as
$\Qinvele_{1,1}$, $\cc{j}$, $\ck{j}$, $\ct{j}$, $\cf{j}$, $\Qinvele_{\ngpts-1,1}$ and $\Qinvele_{\ngpts,1}$, 
and where we have used the relation $\rtre+\rtre^{-1} =8$ to simplify the expressions.

\subsubsection{The inverse %$\Qinvmid$ 
with an even number of grid points $\ngpts$}
%\subsubsection{The first column of $\Qinvmid$}

To make the expressions manageable, we simplify by assuming that $\ngpts$ is an even number. 
%To simplify the expressions, we assume that $\ngpts$ is an even number. 
In this particular case, when solving the $7\times7$ system above, we obtain 
\begin{align*}%\label{j=1odds}
\Qinvele_{1,1}%&=\frac{12}{59}\left(8+3\frac{\modif_{\ngpts-2}}{\nyttplus_{\ngpts}^2}-\frac{9}{59\nyttplus_{\ngpts}^2}+3\frac{24\modif_{\ngpts-2}-5\modif_{\ngpts-1}}{\nyttplus_{\ngpts}^2}\frac{1}{59}-8\frac{8}{59}\right)\\&=\frac{12*3}{59^2\nyttplus_{\ngpts}^2}\left(8*17\nyttplus_{\ngpts}^2-3+83\modif_{\ngpts-2}-5\modif_{\ngpts-1}\right)\\&=\frac{72}{59^2\nyttplus_{\ngpts}^2}\frac{(462+432(\rtre^{\ngpts-3}+\rtre^{3-\ngpts})-39(\rtre^{4-\ngpts}+\rtre^{\ngpts-4}))}{60}\\
=\frac{1}{2}\left(\frac{12\storserien_{\ngpts}}{59\nyttplus_{\ngpts}}\right)^2,&&
\Qinvele_{\ngpts-1,1}%=\frac{12}{59}\left(8+3\frac{\modif_{\ngpts-2}}{\nyttplus_{\ngpts}^2}-\frac{9}{59\nyttplus_{\ngpts}^2}\right),
=\frac{12}{59}\left(\frac{\storserien_{\ngpts}}{\nyttplus_{\ngpts}}-\frac{9}{59\nyttplus_{\ngpts}^2}\right),
&&
\Qinvele_{\ngpts,1}%=\left(8+3\frac{\modif_{\ngpts-2}}{\nyttplus_{\ngpts}^2}\right)\frac{12}{59}
=\frac{12\storserien_{\ngpts}}{59\nyttplus_{\ngpts}}.
\end{align*}
%for $i=1$, $i=\ngpts-1$ and $i=\ngpts$,
where $\storserien_{\ngpts}$ and $\nyttplus_{\ngpts}$ %$\modif_j$ 
are integers given % integers (for integers $j$) described 
in \eqref{nyttplusmm} below.
Note that $\nyttplus_{\ngpts}\geq1$ for even $\ngpts$, so there is no risk of division by zero.
Moreover, 
 we obtain
 \begin{align*}
%\cc{j}=\left(8+3\frac{\modif_{\ngpts-2}}{\nyttplus_{\ngpts}^2}\right)\frac{12}{59},&&
%\ck{j}=\frac{3}{10\nyttplus_{\ngpts}^2}\frac{12}{59},&&
%\ct{j}=\frac{3\rtre^{1.5-\ngpts}}{10\sqrt{6}\nyttplus_{\ngpts}^2}\frac{12}{59},&&
%\cf{j}=\frac{-3\rtre^{\ngpts-1.5}}{10\sqrt{6}\nyttplus_{\ngpts}^2}\frac{12}{59},
\cc{j}=\frac{12\storserien_{\ngpts}}{59\nyttplus_{\ngpts}},&&
\ck{j}=\frac{36}{590\nyttplus_{\ngpts}^2},&&
\ct{j}=\frac{6(\rtre-1)\rtre^{1-\ngpts}}{590\nyttplus_{\ngpts}^2},&&
\cf{j}=\frac{6(\rtre^{-1}-1)\rtre^{\ngpts-1}}{590\nyttplus_{\ngpts}^2},
\end{align*}
which
%Inserting $\call$
inserted
 into the ansatz $\Qinvele_{i,1}=\cc{j}+\ck{j}(-1)^i+\ct{j}\rtre^i+\cf{j}\rtre^{-i}$
leads to
\begin{align*}%\label{j=1mids}
\Qinvele_{i,1}%=\left(8+3\frac{\modif_{\ngpts-2}-\modif_{\ngpts-i}}{\plus_{\ngpts/2-1.5}^2}\right)\frac{12}{59}
%
%=\frac{12}{59}\left(8+3\frac{\modif_{\ngpts-2}-\modif_{\ngpts-i}}{\nyttplus_{\ngpts}^2}\right),&&\text{for } 2\leq i\leq \ngpts-2.
=\frac{12}{59}\left(\frac{\storserien_{\ngpts}}{\nyttplus_{\ngpts}}-\frac{3\modif_{\ngpts-i}}{\nyttplus_{\ngpts}^2}\right),&&\text{for } 2\leq i\leq \ngpts-2.
\end{align*}
%for $2\leq i\leq \ngpts-2$.
%where $\modif_j$, given below, are integers for integers $j$. 
The quantities $\modif_j$ are integers for integers $j$, and are specified below
%\begin{align}\label{nyttplusmm}\begin{split}\nyttplus_{\ngpts}&=\frac{\grund_{\ngpts/2-1}+\grund_{\ngpts/2-2}}{10},\hspace{25.5pt} \storserien_{\ngpts}=\frac{9\grund_{\ngpts/2-1}+4\grund_{\ngpts/2-2}}{10} \\\modif_j&=\frac{\grund_{ j-1}-\grund_{ j-2} -6(-1)^{j}}{60},\hspace{8pt}\rotursym_j=\frac{\frac{1-(-1)^{j}}{2}\grund_{\ngpts/2-1}+\frac{1+(-1)^{j}}{2}\grund_{\ngpts/2-2}-\grund_{\ngpts/2-j}}{60}.\end{split} \end{align} 
\begin{align}\label{nyttplusmm}\begin{split}\nyttplus_{\ngpts}&=\frac{\grund_{\ngpts/2-1}+\grund_{\ngpts/2-2}}{10},\hspace{25.5pt} \storserien_{\ngpts}=\frac{9\grund_{\ngpts/2-1}+4\grund_{\ngpts/2-2}}{10} \\\modif_j&=\frac{\grund_{ j-1}-\grund_{ j-2} -6(-1)^{j}}{60},\hspace{8pt}\rotursym_j=\frac{\frac{1-(-1)^{j}}{2}\grund_{\ngpts/2-1}+\frac{1+(-1)^{j}}{2}\grund_{\ngpts/2-2}-\grund_{\ngpts/2-j}}{60}.\end{split} \end{align} 
where $\grund_j=\rtre^{j}+\rtre^{-j}$. %Note that $\rtre^{-1} +\rtre=8$ leads to $\grund_{j-1}+\grund_{j-1}=8\grund_{j}$.
%\begin{align}
%\grund_j=\rtre^{j}+\rtre^{-j}
% \end{align} 
%$2\sqrt{15}=\sqrt{60}=(\rtre^{0.5}+\rtre^{-0.5})(\rtre^{0.5}-\rtre^{-0.5})=\rtre-\rtre^{-1}$
For convenience, all the $\Qinvele_{i,1}$ presented above will be restated in
\eqref{udda}
and
\eqref{kolonner}, wherein we will also make use of $\rotursym_j$ defined above.

%\subsubsection{The inner columns}

We use the same strategy for the other columns $j>1$.
For %the %inner columns, 
$2\leq j\leq \ngpts-2$, 
we need two different versions of
 the constants $\call$, depending on if we consider $\Qinvele_{i,j}$ for $i\leq j$ or for $i\geq j$.
We let 
 $\Qinvele_{i,j}=\ucc{j}+\uck{j} (-1)^i+\uct{j}\rtre^i+\ucf{j}\rtre^{-i}$ for %$ i\leq j$ 
$2\leq i\leq j\leq \ngpts-2$
 and $\Qinvele_{i,j}=\lcc{j}+\lck{j} (-1)^i+\lct{j}\rtre^i+\lcf{j}\rtre^{-i}$ for %$j\leq i$. 
$2\leq j\leq i\leq \ngpts-2$.
Thus for every $2\leq j\leq \ngpts-2$, we have eight unknown constants, as well as the three remaining unknowns $ \Qinvele_{1,j} $, $ \Qinvele_{\ngpts-1,j}$ and $ \Qinvele_{\ngpts,j}$. The three first and the four last rows in % 
the system above
%\eqref{innercolumns}
 gives us seven conditions. From the rows $i=j-1,j,j+1$, we get three more conditions and in addition, we demand that the two versions of $\Qinvele_{j,j}$ are identical. 
 All in all, this gives a linear system with eleven unknowns
 $ \Qinvele_{1,j} $, $\ucc{j} $, $ \uck{j} $, $ \uct{j} $, $ \ucf{j} $, $ \lcc{j} $, $\lck{j} $, $ \lct{j} $, $\lcf{j} $, $ \Qinvele_{\ngpts-1,j} $ and $ \Qinvele_{\ngpts,j} $
 and eleven conditions.

%Again %limiting ourselves to consider the case where $\ngpts$ is even, 
We still
consider %simplify by considering
even numbers of $\ngpts$. %thereafter 
Solving for the unknowns and inserting $\callu$ and $\calll$ into their respective ansatz, 
we eventually end up with $\Qinvele_{i,j}$ for the inner columns, presented below in \eqref{rader} and \eqref{mitten}. % for $2\leq j\leq \ngpts-2$.
Furthermore, repeating the procedure for the last two columns, we obtain $\Qinvele_{i,j}$ for $j=\ngpts-1$ and $j=\ngpts$, given in \eqref{udda}
and
\eqref{kolonner}.
To simplify the expressions in \eqref{kolonner}, \eqref{rader} and \eqref{mitten}, we have used $\rotursym_j$ in \eqref{nyttplusmm}. %\eqref{roturdef}

In summary, when $\ngpts$ is even, the inverse of $\Qbar$ is given by 
$\Qbar^{-1}= (\Qinvele_{i,j})_{\ngpts\times\ngpts}$ with $\Qinvele_{i,j}$ as described in
\eqref{udda},
 \eqref{rader},
\eqref{kolonner} and
 \eqref{mitten} below. First, the corner elements are
\begin{align}\label{udda}
\Qinvele_{1,1}
&=\frac{72\storserien_{\ngpts}^2}{59^2\nyttplus_{\ngpts}^2},\hspace{18pt}
\Qinvele_{1,\ngpts-1}%=\frac{12(12\storserien_{\ngpts}^2+9-59\storserien_{\ngpts}\nyttplus_{\ngpts})}{59^2\nyttplus_{\ngpts}^2},
=\frac{12}{59}\left(\frac{12\storserien_{\ngpts}^2+9}{59\nyttplus_{\ngpts}^2}-\frac{\storserien_{\ngpts}}{\nyttplus_{\ngpts}}\right),
\hspace{18pt}
\Qinvele_{1,\ngpts}
=-\frac{12\storserien_{\ngpts}}{59\nyttplus_{\ngpts}},\notag\\
\Qinvele_{\ngpts-1,1}
&=\frac{12}{59}\left(\frac{\storserien_{\ngpts}}{\nyttplus_{\ngpts}}-\frac{9}{59\nyttplus_{\ngpts}^2}\right),\hspace{15pt}\Qinvele_{\ngpts-1,\ngpts-1}
=\frac{72\storserien_{\ngpts}^2}{59^2\nyttplus_{\ngpts}^2},\hspace{15pt}\Qinvele_{\ngpts-1,\ngpts}
=-\frac{12\storserien_{\ngpts}}{59\nyttplus_{\ngpts}},
\\
\Qinvele_{\ngpts,1}
&=\frac{12\storserien_{\ngpts}}{59\nyttplus_{\ngpts}},\hspace{30pt}\Qinvele_{\ngpts,\ngpts-1}=\frac{12\storserien_{\ngpts}}{59\nyttplus_{\ngpts}},\hspace{30pt}\Qinvele_{\ngpts,\ngpts}=0.\notag
\end{align}
%for the corner elements of the inverse matrix.
For $2\leq i\leq \ngpts-2$, we obtain
\begin{align}\label{kolonner}\begin{split}
\Qinvele_{i,1}
=\frac{12}{59}\left(\frac{\storserien_{\ngpts}}{\nyttplus_{\ngpts}}-\frac{3\modif_{\ngpts-i}}{\nyttplus_{\ngpts}^2}\right),
 \hspace{10pt}
%\Qinvele_{i,\ngpts-1}= 12^2\frac{\storserien_{\ngpts}\rotursym_i}{59\nyttplus_{\ngpts}^2}-12\left(\frac{\rotursym_i}{\nyttplus_{\ngpts}}-\frac{3\modif_{i}}{59\nyttplus_{\ngpts}^2}\right),
\Qinvele_{i,\ngpts-1}= 36\frac{4\storserien_{\ngpts}\rotursym_i+\modif_{i}}{59\nyttplus_{\ngpts}^2}-\frac{12\rotursym_i}{\nyttplus_{\ngpts}},
\hspace{10pt}
\Qinvele_{i,\ngpts}
=\frac{-12\rotursym_i}{ \nyttplus_{\ngpts}}, 
 \end{split}
 \end{align}
while we for $2\leq j\leq \ngpts-2$ have
\begin{align}\label{rader}%\begin{split}
 \Qinvele_{1,j} 
%& = 12^2\frac{\storserien_{\ngpts}\rotursym_j}{59\nyttplus_{\ngpts}^2}-\frac{12}{59}\left(\frac{\storserien_{\ngpts}}{\nyttplus_{\ngpts}}-3\frac{\modif_{\ngpts-j}}{\nyttplus_{\ngpts}^2}\right),
& = 36\frac{4\storserien_{\ngpts}\rotursym_j+\modif_{\ngpts-j}}{59\nyttplus_{\ngpts}^2}-\frac{12\storserien_{\ngpts}}{59\nyttplus_{\ngpts}},
 &&
 \Qinvele_{\ngpts-1,j}=%12\left(\frac{\rotursym_j}{\nyttplus_{\ngpts}}-\frac{3\modif_{j}}{59\nyttplus_{\ngpts}^2}\right),
\frac{12\rotursym_j}{\nyttplus_{\ngpts}}-\frac{36\modif_{j}}{59\nyttplus_{\ngpts}^2},
 &&%\hspace{30pt}
 \Qinvele_{\ngpts,j} =\frac{12\rotursym_j}{\nyttplus_{\ngpts}}.
% \end{split}
 \end{align}
Finally, the interior elements %columns and rows 
are
\begin{align}\label{mitten}\begin{split}
\Qinvele_{i,j}
&=12^2\frac{\rotursym_i \rotursym_j}{\nyttplus_{\ngpts}^2}-12\left(\frac{\rotursym_i}{ \nyttplus_{\ngpts}}-\frac{ \modif_i\modif_{ \ngpts-j}}{ \nyttplus_{\ngpts}^2} \right),\hspace{30pt}\text{for }2\leq i\leq j\leq \ngpts-2,
\\
 \Qinvele_{i,j}
 &=12\left(\frac{\rotursym_j}{ \nyttplus_{\ngpts}}-\frac{ \modif_j\modif_{ \ngpts-i}}{ \nyttplus_{\ngpts}^2} \right),\hspace{86pt}\text{for }2\leq j\leq i\leq \ngpts-2.
 \end{split}
 \end{align}
In the expressions above we have used $\nyttplus_{\ngpts}$, $\storserien_{\ngpts}$, $\modif_j$
and $\rotursym_j$ defined in \eqref{nyttplusmm}.
Next, we recall the structure in \eqref{Qparts}, and identify $\qvec$ in \eqref{Qtildeinv4} as
\begin{align*}
\qvec^\trans=
\left[\begin{array}{cccccccccc}%{cccc|ccc|cccc}
 \frac{59}{96} & -\frac{1}{12} & -\frac{1}{32} & 0 & 0&\cdots&0 \\
 \end{array}\right],
 \end{align*}
and compute
$\qvec^\trans\Qbar^{-1}$ as $(\qvec^\trans\Qbar^{-1})_j= \frac{59}{96}\Qinvele_{1,j} -\frac{1}{12} \Qinvele_{2,j}-\frac{1}{32}\Qinvele_{3,j}$. This gives
\begin{align}\label{qTQinv4}
(\qvec^\trans\Qbar^{-1})_j=\left\{\begin{array}{ll} \vspace{3pt}
\frac{12\storserien_{\ngpts}}{59\nyttplus_{\ngpts}}-1&\text{for }j=1\\\vspace{3pt}
\frac{ 12\rotursym_j}{\nyttplus_{\ngpts}}-1
&\text{for }j=2, \hdots,\ngpts-2\\\vspace{3pt}%checked
\frac{12\storserien_{\ngpts}}{59\nyttplus_{\ngpts}}-1&\text{for }j=\ngpts-1\\
-1&\text{for }j=\ngpts,
\end{array}\right.
 \end{align}
where we have used %simplified the expressions above using 
the structures of $\Qinvele_{i,j}$ in \eqref{udda},
\eqref{kolonner},
 \eqref{rader} and
 \eqref{mitten},
together with the following relations:
\begin{align*}
\modif_2&=0,&\rotursym_2&=0,&\modif_{\ngpts-2}&=\frac{\storserien_{\ngpts}\nyttplus_{\ngpts}-8\nyttplus_{\ngpts}^2}{3}\\
\modif_3&=1,&\rotursym_3&=\frac{\storserien_{\ngpts}-8\nyttplus_{\ngpts}}{3},&
\modif_{\ngpts-3}&=\frac{-2\storserien_{\ngpts}^2+33\storserien_{\ngpts}\nyttplus_{\ngpts}-136\nyttplus_{\ngpts}^2}{3}.
\end{align*}
Inserting $\Qbar^{-1}$ from \eqref{udda},
\eqref{kolonner},
 \eqref{rader} and
 \eqref{mitten}, and \eqref{qTQinv4} into \eqref{QsnokInv} and \eqref{GreenEtc} yields the inverse of $\Qsnok$ in the (4,2) order accurate case (for $\ngpts$ even).

\section{Proof of Theorem~\ref{ThmGenInv}}% and Corollary~\ref{CorGenInv}}
\label{ProofThmInv}

Theorem~\ref{ThmGenInv} states that the inverse of $\OPsnok$ from \eqref{Kop2} is equal to the expression \eqref{OPsnokinvRobin}. 
This is shown in Section~\ref{A2}, however,
first, we present some useful relations.

\subsection{Preliminaries}
\label{Preliminaries}

Note that $D_2\ett=0$ and $D_2\xfet=0$, since $D_2$ approximates the second derivative operator 
(these two relations actually hold also for the inconsistent (2,0) order accurate operators in Sections~\ref{Inverse20} and \ref{AppWide}).
Furthermore, 
 $\dselr^\trans$ 
 consistently approximate the first derivative, so that $\dselr^\trans\ett=0$ and $\dselr^\trans\xfet=1$.
 Hence
 \begin{align}\label{Sconsistent}
 \dsel^\trans(\domsize\ett-\xfet)=-1,&&
 \dsel^\trans\xfet=1,&&
 \dser^\trans(\domsize\ett-\xfet)=-1,&&
 \dser^\trans\xfet=1.
 \end{align}
Combining the above relations with $A=-\PH D_2+\er\dser^\trans-\el\dsel^\trans$ from \eqref{SBPprop2},
gives 
% $A\ett=\noll$ and $A\xfet=\er-\el$ 
 \begin{align}\label{AxA1-x}
 A(\domsize\ett-\xfet)=\el-\er,&&A\xfet=\er-\el.
 \end{align}
Now, we define 
the additional $(\ngpts-1)\times1$-vectors $\one=[1\ 1\ \hdots\ 1]^\trans$ and $\xvec=h[1\ 2\ \hdots\ \ngpts-1]^\trans$ (they are 
 shorter versions of $\ett$ and $\xfet$ in %\eqref{ettochx}). 
 Theorem~\ref{ThmGenInv}).
% With these new variables and with the notation from \eqref{ASparts}, the relations \eqref{AxA1-x} lead to
%\begin{align*}
%%\left[\begin{array}{ccc}\Aon&\Awest^\trans&\Aoff\\\Awest&\Amid&\Aeast\\\Aoff&\Aeast^\trans&\Aon\end{array}\right]\left[\begin{array}{c}1\\\one-\xvec\\0\\\end{array}\right]=
%\left[\begin{array}{c} \Aon+\Awest^\trans( \one-\xvec)\\ \Awest+\Amid( \one-\xvec)\\ \Aoff+\Aeast^\trans( \one-\xvec)\end{array}\right]=\left[\begin{array}{c}1\\\vec{0}\\-1\end{array}\right],&&
%%\left[\begin{array}{ccc}\Aon&\Awest^\trans&\Aoff\\\Awest&\Amid&\Aeast\\\Aoff&\Aeast^\trans&\Aon\end{array}\right]\left[\begin{array}{c}0\\\xvec\\1\\\end{array}\right]=
%\left[\begin{array}{c}\Awest^\trans\xvec+ \Aoff\\\Amid\xvec+ \Aeast\\\Aeast^\trans\xvec+ \Aon\end{array}\right]=\left[\begin{array}{c}-1\\\vec{0}\\1\end{array}\right]
%\end{align*}
 With these new variables and with the notation from \eqref{Aparts}, the relations \eqref{AxA1-x} can be expressed as
\begin{align*}
\left[\begin{array}{c} \domsize\Abeg+\Awest^\trans( \domsize\one-\xvec)\\ \domsize\Awest+\Amid(\domsize \one-\xvec)\\ \domsize\Aoff+\Aeast^\trans(\domsize \one-\xvec)\end{array}\right]=\left[\begin{array}{c}1\\\vec{0}\\-1\end{array}\right],&&
\left[\begin{array}{c}\Awest^\trans\xvec+ \domsize\Aoff\\\Amid\xvec+\domsize \Aeast\\\Aeast^\trans\xvec+ \domsize\Aend\end{array}\right]=\left[\begin{array}{c}-1\\\vec{0}\\1\end{array}\right].
\end{align*}
%Assuming that $A$ is correctly constructed such that $\Amid$ in invertible, see Remark~\ref{Abarinverterbar}, this leads to the following relations:
Given that $A$ is correctly constructed, such that $\Amid$ in invertible, this leads to the relations
\begin{align}\label{xuttrycktaiA}
\one-\xvec/\domsize &=-\Amid^{-1}\Awest,&\xvec/\domsize &=-\Amid^{-1}\Aeast
\end{align}and
\begin{align}\label{Bonus}
%\Aon&=\Awest^\trans\Amid^{-1}\Awest+1 =\Aeast^\trans\Amid^{-1}\Aeast+1 ,&\Aoff&=\Aeast^\trans\Amid^{-1}\Awest-1 =\Awest^\trans\Amid^{-1}\Aeast-1 .
\Abeg&=\Awest^\trans\Amid^{-1}\Awest+\frac{1}{\domsize},&\Aend&=\Aeast^\trans\Amid^{-1}\Aeast+\frac{1}{\domsize},&\Aoff&=\Aeast^\trans\Amid^{-1}\Awest-\frac{1}{\domsize} =\Awest^\trans\Amid^{-1}\Aeast-\frac{1}{\domsize}.
\end{align}
Now, multiplying $A$ from \eqref{Aparts} by $\GreenDisc$ from \eqref{InverseParts} and using the relations \eqref{xuttrycktaiA}, we get
 \begin{align}\label{AG}
 A\GreenDisc=\left[\begin{array}{ccc}0&\Awest^\trans\Amid^{-1}&0\\0&\Imid&0\\0&\Aeast^\trans\Amid^{-1}&0\end{array}\right]
% =\left[\begin{array}{ccc}0&(\xvec/\domsize-\one)^\trans&0\\0&\Imid&0\\0&-\xvec^\trans/\domsize&0\end{array}\right]
 =I-\el(\ett-\xfet/\domsize )^\trans-\er\xfet^\trans/\domsize 
 \end{align}where $\Imid$ is the $(\ngpts-1)\times(\ngpts-1)$ identity matrix. 
 From \eqref{InverseParts} we have 
 $
\pickupl=\ett-\xfet/\domsize -\GreenDisc\dsel$ and $\pickupr=\xfet/\domsize +\GreenDisc\dser
$, and
using the relations \eqref{AxA1-x}, \eqref{AG} and \eqref{Sconsistent}, we arrive at
\begin{align}\label{Apickup}
A\pickupl=-\dsel,&&A\pickupr=\dser.
\end{align}
The vectors $\elr$ picks out the first and last elements in the vectors they are multiplied by, such that 
 \begin{align}\label{relelr}
 \begin{aligned}
 \el^\trans(\ett-\xfet/\domsize )=1,&&
 \el^\trans\xfet/\domsize =0,&&
 \er^\trans(\ett-\xfet/\domsize )=0,&&
 \er^\trans\xfet /\domsize=1,\\
\el^\trans\pickupl=1,&&
\el^\trans\pickupr=0, &&
\er^\trans\pickupl=0,&&
\er^\trans\pickupr=1.
 \end{aligned}
 \end{align}
Finally, from \eqref{InverseParts} we have
 \begin{align}\label{Grels}
\el^\trans\GreenDisc=\er^\trans\GreenDisc=\noll^\trans,&&\dsel^\trans \GreenDisc=(\ett-\xfet/\domsize -\pickupl)^\trans,&&\dser^\trans \GreenDisc=(\pickupr-\xfet/\domsize )^\trans.
 \end{align}
% $\el^\trans\GreenDisc=\er^\trans\GreenDisc=\noll^\trans$ and 
We are now ready to prove Theorem~\ref{ThmGenInv}.

%\subsection{Confirmation of $\OPsnok^{-1}$ in Equation~\eqref{OPsnokinvRobin}}
\subsection{Confirmation of Equation~\eqref{OPsnokinvRobin} with \eqref{InverseParts},
\eqref{determining} and
\eqref{qsnokdefALT}}
\label{A2}

We multiply $\OPsnok$ in %the second row of \eqref{Kop}
\eqref{Kop2}
 by the expression for $\OPsnok^{-1}$ in \eqref{OPsnokinvRobin}, with the aim of showing that $\OPsnok\OPsnok^{-1}=I$ indeed holds. In the first step, \eqref{OPsnokinvRobin} yields 
\begin{align}\label{AsnokAsnokinv}
\OPsnok\OPsnok^{-1}&=\OPsnok\GreenDisc+\underbrace{\OPsnok\left[\begin{array}{c}-\sigmal\pickupl^\trans\\ -\sigmar\pickupr^\trans\\(\ett-\xfet/\domsize )^\trans\\ \xfet^\trans /\domsize\end{array}\right]^{\hspace{-3pt}\trans}\hspace{-2pt}}_{\tmpvar}\hspace{2pt} \detmatris^{-1} \left[\begin{array}{c}\pickupl^\trans\\ \pickupr^\trans\\\bl(\ett-\xfet/\domsize )^\trans\\ \br\xfet^\trans/\domsize \end{array}\right].
\end{align}
We start by looking at the first term in \eqref{AsnokAsnokinv}. First using \eqref{Kop2}, followed by the relations in \eqref{AG} and \eqref{Grels}, and thereafter just rearranging the terms, we arrive at\newpage
\begin{align}\label{AAinv1}\begin{split}
\OPsnok\GreenDisc&=A\GreenDisc-\hspace{-1pt}\left[\hspace{-5pt}\begin{array}{c}\el^\trans\\-\dsel^\trans\end{array}\hspace{ -4pt}\right]^\trans\left[\hspace{ -4pt}\begin{array}{cc}\taul\al&\taul\bl+1\\\sigmal\al&\sigmal\bl\end{array}\hspace{ -4pt}\right]\left[\hspace{-5pt}\begin{array}{c}\el^\trans\\-\dsel^\trans\end{array}\hspace{ -4pt}\right]\GreenDisc\\&\hspace{40pt}-\hspace{-1pt}\left[\hspace{ -4pt}\begin{array}{c}\er^\trans\\\dser^\trans\end{array}\hspace{ -4pt}\right]^\trans\left[\hspace{ -4pt}\begin{array}{cc}\taur\ar&\taur\br+1\\\sigmar\ar&\sigmar\br\end{array}\hspace{ -4pt}\right]\left[\hspace{ -4pt}\begin{array}{c}\er^\trans\\\dser^\trans\end{array}\hspace{ -4pt}\right]\GreenDisc\\
%&=I-\el(\ett-\xfet)^\trans-\er\xfet^\trans\\&-\left[\hspace{-5pt}\begin{array}{c}\el^\trans\\-\dsel^\trans\end{array}\hspace{-3pt}\right]^\trans\left[\hspace{-3pt}\begin{array}{cc}\taul\al&\taul\bl+1\\\sigmal\al&\sigmal\bl\end{array}\hspace{-3pt}\right]\left[\hspace{-5pt}\begin{array}{c}\noll^\trans\\-(\ett-\xfet-\pickupl)^\trans\end{array}\hspace{-3pt}\right]-\left[\hspace{-3pt}\begin{array}{c}\er^\trans\\\dser^\trans\end{array}\hspace{-3pt}\right]^\trans\left[\hspace{-3pt}\begin{array}{cc}\taur\ar&\taur\br+1\\\sigmar\ar&\sigmar\br\end{array}\hspace{-3pt}\right]\left[\hspace{-3pt}\begin{array}{c}\noll^\trans\\\pickupr^\trans-\xfet^\trans\end{array}\hspace{-3pt}\right]\\
&=I-\el(\ett-\xfet/\domsize)^\trans-\er\xfet^\trans/\domsize-\left[\hspace{-5pt}\begin{array}{c}\el^\trans\\-\dsel^\trans\end{array}\hspace{-3pt}\right]^\trans\left[\hspace{-3pt}\begin{array}{c}\taul\bl+1\\\sigmal\bl\end{array}\hspace{-3pt}\right]\left(\pickupl-\ett+\xfet/\domsize\right)^\trans\\&\hspace{150pt}-\left[\hspace{-3pt}\begin{array}{c}\er^\trans\\\dser^\trans\end{array}\hspace{-3pt}\right]^\trans\left[\hspace{-3pt}\begin{array}{c}\taur\br+1\\\sigmar\br\end{array}\hspace{-3pt}\right]\left(\pickupr-\xfet/\domsize\right)^\trans\\
%&=I-\left[\hspace{-5pt}\begin{array}{c}\el^\trans\\\er^\trans\\-\dsel^\trans\\\dser^\trans\end{array}\hspace{-3pt}\right]^\trans\left[\hspace{-5pt}\begin{array}{cccc}\taul\bl+1&0&-\taul&0\\0&\taur\br+1&0&-\taur\\\sigmal\bl&0&-\sigmal&0\\0&\sigmar\br&0&-\sigmar\end{array}\hspace{-3pt}\right]\left[\begin{array}{c}\pickupl^\trans\\ \pickupr^\trans\\\bl(\ett-\xfet)^\trans\\ \br\xfet^\trans\end{array}\right]\\
&=I-\left[\hspace{-5pt}\begin{array}{c}(\taul\bl+1)\el^\trans-\sigmal\bl\dsel^\trans\\(\taur\br+1)\er^\trans+\sigmar\br\dser^\trans\\-\left(\taul\el^\trans-\sigmal\dsel^\trans\right)\\-(\taur\er^\trans+\sigmar\dser^\trans)\end{array}\hspace{-3pt}\right]^\trans\left[\begin{array}{c}\pickupl^\trans\\ \pickupr^\trans\\\bl(\ett-\xfet/\domsize)^\trans\\ \br\xfet^\trans/\domsize\end{array}\right].\end{split}
\end{align}
%Next, we take a look at the second term in \eqref{AsnokAsnokinv}, or rather the part $\tmpvar$ multiplying $\detmatris^{-1}$ from the left, to be more precise. 
Next, we look at the part $\tmpvar$ in \eqref{AsnokAsnokinv}.
After rewriting $\OPsnok$ using \eqref{Kop2}, we use the relations in \eqref{Apickup}, \eqref{AxA1-x}, 
\eqref{relelr}, \eqref{Sconsistent} and \eqref{qsnokdefALT}. Thereafter, the resulting terms are rearranged. These steps are shown below in \eqref{AAinv2}.%\newpage
\begin{align}\label{AAinv2}\begin{split}
\tmpvar%&=\OPsnok\left[\hspace{-3pt}\begin{array}{c}-\sigmal\pickupl^\trans\\ -\sigmar\pickupr^\trans\\(\ett-\xfet)^\trans\\ \xfet^\trans\end{array}\hspace{-3pt}\right]^\trans \\
&=A\left[\hspace{-3pt}\begin{array}{c}-\sigmal\pickupl^\trans\\ -\sigmar\pickupr^\trans\\(\ett-\xfet/\domsize)^\trans\\ \xfet^\trans/\domsize\end{array}\hspace{-3pt}\right]^\trans-\left[\hspace{-5pt}\begin{array}{c}\el^\trans\\-\dsel^\trans\end{array}\hspace{-3pt}\right]^\trans\left[\hspace{-3pt}\begin{array}{cc}\taul\al&\taul\bl+1\\\sigmal\al&\sigmal\bl\end{array}\hspace{-3pt}\right]\left[\hspace{-5pt}\begin{array}{c}\el^\trans\\-\dsel^\trans\end{array}\hspace{-3pt}\right]\left[\hspace{-3pt}\begin{array}{c}-\sigmal\pickupl^\trans\\ -\sigmar\pickupr^\trans\\(\ett-\xfet/\domsize)^\trans\\ \xfet^\trans/\domsize\end{array}\hspace{-3pt}\right]^\trans\\&\hspace{90pt}-\left[\hspace{-3pt}\begin{array}{c}\er^\trans\\\dser^\trans\end{array}\hspace{-3pt}\right]^\trans\left[\hspace{-3pt}\begin{array}{cc}\taur\ar&\taur\br+1\\\sigmar\ar&\sigmar\br\end{array}\hspace{-3pt}\right]\left[\hspace{-3pt}\begin{array}{c}\er^\trans\\\dser^\trans\end{array}\hspace{-3pt}\right]\left[\begin{array}{c}-\sigmal\pickupl^\trans\\ -\sigmar\pickupr^\trans\\(\ett-\xfet/\domsize)^\trans\\ \xfet^\trans/\domsize\end{array}\right]^\trans\\
&=\left[\hspace{-3pt}\begin{array}{c}\sigmal\dsel^\trans\\ -\sigmar\dser^\trans\\(\el^\trans-\er^\trans)/\domsize\\ (\er^\trans-\el^\trans)/\domsize\end{array}\hspace{-3pt}\right]^\trans-\left[\hspace{-5pt}\begin{array}{c}\el^\trans\\-\dsel^\trans\end{array}\hspace{-3pt}\right]^\trans\left[\hspace{-3pt}\begin{array}{cc}\taul\al&\taul\bl+1\\\sigmal\al&\sigmal\bl\end{array}\hspace{-3pt}\right]\left[\hspace{-5pt}\begin{array}{cccc}-\sigmal&0&1&0\\-\sigmal\qsnokl&\sigmar\qsnokc&1/\domsize&-1/\domsize\end{array}\hspace{-3pt}\right]\\&\hspace{77pt}-\left[\hspace{-3pt}\begin{array}{c}\er^\trans\\\dser^\trans\end{array}\hspace{-3pt}\right]^\trans\left[\hspace{-3pt}\begin{array}{cc}\taur\ar&\taur\br+1\\\sigmar\ar&\sigmar\br\end{array}\hspace{-3pt}\right]\left[\hspace{-5pt}\begin{array}{cccc}0&-\sigmar&0&1\\\sigmal\qsnokc&-\sigmar\qsnokr&-1/\domsize&1/\domsize\end{array}\hspace{-3pt}\right]\\
&=\left[\hspace{-5pt}\begin{array}{c}(\taul\bl+1)\el^\trans-\sigmal\bl\dsel^\trans\\(\taur\br+1)\er^\trans+\sigmar\br\dser^\trans\\-\taul\el^\trans+\sigmal\dsel^\trans\\-\taur\er^\trans-\sigmar\dser^\trans\end{array}\hspace{-3pt}\right]^\trans\left[\hspace{-5pt}\begin{array}{cccc}\taul+\sigmal\qsnokl&-\sigmar\qsnokc&0&0\\-\sigmal\qsnokc&\taur+\sigmar\qsnokr&0&0\\\duall&0&\al+\frac{\bl}{\domsize}&-\frac{\bl}{\domsize}\\0&\dualr&-\frac{\br}{\domsize}&\ar+\frac{\br}{\domsize}\end{array}\hspace{-3pt}\right],
\end{split}
\end{align}
We note that the last $4\times4$-matrix is nothing but $\detmatris$ from \eqref{determining}.
Inserting the results from \eqref{AAinv1} and \eqref{AAinv2} into \eqref{AsnokAsnokinv} gives us
\begin{align*}
\OPsnok\OPsnok^{-1}&=
I-\left[\hspace{-5pt}\begin{array}{c}(\taul\bl+1)\el^\trans-\sigmal\bl\dsel^\trans\\(\taur\br+1)\er^\trans+\sigmar\br\dser^\trans\\-\taul\el^\trans+\sigmal\dsel^\trans\\-\taur\er^\trans-\sigmar\dser^\trans\end{array}\hspace{-3pt}\right]^\trans\left[\begin{array}{c}\pickupl^\trans\\ \pickupr^\trans\\\bl(\ett-\xfet /\domsize)^\trans\\ \br\xfet^\trans/\domsize \end{array}\right]\\&
+\left[\hspace{-5pt}\begin{array}{c}(\taul\bl+1)\el^\trans-\sigmal\bl\dsel^\trans\\(\taur\br+1)\er^\trans+\sigmar\br\dser^\trans\\-\taul\el^\trans+\sigmal\dsel^\trans\\-\taur\er^\trans-\sigmar\dser^\trans\end{array}\hspace{-3pt}\right]^\trans \detmatris \detmatris^{-1} \left[\begin{array}{c}\pickupl^\trans\\ \pickupr^\trans\\\bl(\ett-\xfet /\domsize)^\trans\\ \br\xfet^\trans/\domsize \end{array}\right]=I,
\end{align*}
concluding the proof.

\section{Proofs of the relations between $\qhattot$, $\app$, $\qtot$ and $\qsnoktot$}

Below we present the proofs of Theorem~\ref{happ=1/qsnok} and 
the
Lemmas~\ref{happ=1/q} and
\ref{qsnoksame}.

\subsection{Proof of Theorem~\ref{happ=1/qsnok}}%that $h\app=1/\qsnok$}
\label{happqsnok=1}

We aim to relate $\app$ in \eqref{BorrowProp} to $\qhattot$ in \eqref{qhat}.
Note that the latter quantity relies on that $\qhatl=\qhatr$ in \eqref{qsnokdefALT}. To emphasize this, we introduce $\qhatd=\qhatlr$.

We start by defining 
$\widetilde{\scv}=\scv-\pickupl\arbscall +\pickupr\arbscalr$ with $\pickuplr$ from
\eqref{InverseParts},
and compute
\begin{align}\label{helpsnok}
\widetilde{\scv}^\trans A \widetilde{\scv}%&=\scv^\trans A\scv-2\arbscall\scv^\trans A\pickupl +2\arbscalr\scv^\trans A\pickupr\\&+\arbscall^2\pickupl^\trans A\pickupl -2\arbscall\arbscalr\pickupl^\trans A\pickupr +\arbscalr^2\pickupr^\trans A \pickupr\\
&=\scv^\trans A\scv+2\arbscall\scv^\trans \dsel +2\arbscalr\scv^\trans \dser+\arbscall^2\qhatl +2\arbscall\arbscalr\qhatc +\arbscalr^2\qhatr
\end{align}
using \eqref{Apickup} and \eqref{qsnokdefALT}. The $(\ngpts+1)\times1$-vector $\scv$ is arbitrary and
for the scalars $\arbscallr$ we make the ansatz
$\arbscall=(\mainl\dsel^\trans +\minorr\dser^\trans)\scv$ and $\arbscalr=(\mainr\dser^\trans +\minorl\dsel^\trans)\scv$ where $\minorlr$ and $\mainlr$ are scalars yet to be determined. Inserted into \eqref{helpsnok}, this yields
\begin{align}\label{insattansatz}
\widetilde{\scv}^\trans A \widetilde{\scv}=\scv^\trans A\scv+\scv^\trans(\zl \dsel\dsel^\trans +2\zc\dsel\dser^\trans +\zr \dser\dser^\trans )\scv 
\end{align}
 where we have defined 
\begin{align}\label{z0Nc}\begin{split}
\zl&=2\mainl+2\qhatc \mainl\minorl+\qhatl\mainl^2+\qhatr\minorl^2\\
\zr&=2\mainr+2\qhatc \mainr\minorr+\qhatr\mainr^2+\qhatl\minorr^2\\
\zc&=\minorl+\minorr+\qhatl\mainl\minorr+\qhatr\mainr\minorl+\qhatc \mainl\mainr+\qhatc \minorl\minorr.\end{split}
\end{align}
Using the "borrowing technique", $\app$ is the maximum value such that $\Atilde\geq0$ still holds, referring to $\app$ and $\Atilde$ from \eqref{BorrowProp}.
For \eqref{insattansatz} to correspond to \eqref{BorrowProp}, we need $\zl=\zr$ 
and $\zc=0$, and under these constraints we must mimimize $\zlr$.
To get there, we first define
$\xl =\mainl+\minorl$, $\yl =\mainl-\minorl$, $\xr =\mainr+\minorr$ and $\yr =\mainr-\minorr$.
Now
%\begin{align*}
%2\mainlr&=x_{0,\ngpts}+y_{0,\ngpts},&
%x_0x_\ngpts&=\mainl\mainr+\minorl\minorr+\mainl\minorr+\minorl\mainr,\\
%2\minorlr&=x_{0,\ngpts}-y_{0,\ngpts},&
%y_0y_\ngpts&=\mainl\mainr+\minorl\minorr-\mainl\minorr-\minorl\mainr.
%\end{align*}
%\begin{align*}
%\xlr +\ylr &=2\mainlr,&
%\xlr ^2-\ylr ^2&=4\mainlr\minorlr,&
%\xlr ^2+\ylr ^2&=2(\mainlr^2+\minorlr^2).
%\end{align*}
\begin{align*}
\xl +\yl &=2\mainl,&
\xl ^2-\yl ^2&=4\mainl\minorl,&
\xl ^2+\yl ^2&=2(\mainl^2+\minorl^2),\\
\xr +\yr &=2\mainr,&
\xr ^2-\yr ^2&=4\mainr\minorr,&
\xr ^2+\yr ^2&=2(\mainr^2+\minorr^2).
\end{align*}
Inserted into $\zl$ and $\zr$ in \eqref{z0Nc}, these relations gives us %(under the assumption that )
\begin{align*}
\zlr&=\xlr +\ylr +\qhatc \frac{\xlr ^2-\ylr ^2}{2}+\qhatd \frac{\xlr ^2+\ylr ^2}{2}\\
%&=x_\indexlr+y_\indexlr+x_\indexlr^2\frac{\qd +\qhatc }{2}+y_\indexlr^2\frac{\qd -\qhatc }{2}\\
&=\frac{\qhatd +\qhatc }{2}\left(\xlr +\frac{1}{\qhatd +\qhatc }\right)^2+\frac{\qhatd -\qhatc }{2}\left(\ylr +\frac{1}{\qhatd -\qhatc }\right)^2-\frac{\qhatd }{\qhatd ^2-\qhatc ^2}
\end{align*}
where we have used that $\qhatd=\qhatl=\qhatr$.
Note that for fixed values of $\zl$ and $\zr$, the pairs %$\xlr $, $\ylr $ 
$(\xl , \yl )$ and $(\xr , \yr )$
describe ellipses. 
Reformulated in a parametric form, they are
\begin{align}\label{xochydiffr}
%\begin{split}
%\xlr &=-\frac{1}{\qhatd +\qhatc }+\sqrt{\frac{2}{\qhatd +\qhatc }}\ \rlr\cos(\theta_\indexlr),\\
%\ylr &=-\frac{1}{\qhatd -\qhatc }+\sqrt{\frac{2}{\qhatd -\qhatc }}\ \rlr\sin(\theta_\indexlr),\end{split}
\begin{split}
\xl &=\frac{-1}{\qhatd +\qhatc }+\sqrt{\frac{2}{\qhatd +\qhatc }}\ \rl\cos(\theta_\indexl),\hspace{17pt}\yl =\frac{-1}{\qhatd -\qhatc }+\sqrt{\frac{2}{\qhatd -\qhatc }}\ \rl\sin(\theta_\indexl),\\
\xr &=\frac{-1}{\qhatd +\qhatc }+\sqrt{\frac{2}{\qhatd +\qhatc }}\ \rr\cos(\theta_\indexr),\hspace{15pt}\yr =\frac{-1}{\qhatd -\qhatc }+\sqrt{\frac{2}{\qhatd -\qhatc }}\ \rr\sin(\theta_\indexr),\end{split}
\end{align}
%leading to
%%\begin{align*}
%%%\zlr&=\rlr^2-\frac{\qd }{\qd ^2-\qhatc ^2}.
%%z_{0}&=\rl^2-\frac{\qd }{\qd ^2-\qhatc ^2},&z_{\ngpts}&=\rr^2-\frac{\qd }{\qd ^2-\qhatc ^2}.
%%\end{align*}
%$\zl=\rl^2-\qhatd /(\qhatd ^2-\qhatc ^2)$ and $\zr=\rr^2-\qhatd /(\qhatd ^2-\qhatc ^2)$.
where
$\rl^2=\zl+\qhatd /(\qhatd ^2-\qhatc ^2)$ and $\rr^2=\zr+\qhatd /(\qhatd ^2-\qhatc ^2)$.
To enforce $\zl=\zr$, we simply let $\rl=\rr=r$.
This gives us
\begin{align}\label{zlrsamma}
\zlr
&=r^2-\frac{\qhatd }{\qhatd ^2-\qhatc ^2}.
\end{align}
Next, we need to fulfull the requirement $\zc=0$. Inserting the relations
\begin{align*}
\minorlr=\frac{\xlr -\ylr }{2},&&
\mainl\minorr+\minorl\mainr=\frac{\xl \xr-\yl \yr }{2},&&
\mainl\mainr+\minorl\minorr=\frac{\xl \xr+\yl \yr }{2}
\end{align*}
 into $\zc$ in \eqref{z0Nc}, and thereafter using \eqref{xochydiffr} with $\rlr=r$,
leads to
\begin{align*}
2\zc&=\xl -\yl +\xr -\yr +\qhatd (\xl \xr -\yl \yr )+\qhatc (\xl \xr +\yl \yr )\\%&=x_{0}-y_{0}+x_{N}-y_{N}+x_0x_N(\qd +\qhatc )-y_0y_N(\qd -\qhatc )\\&=-\frac{2}{\qd +\qhatc }+\sqrt{\frac{2}{\qd +\qhatc }}\ r(\cos(\theta_\indexl)+\cos(\theta_{N}))\\&+\frac{2}{\qd -\qhatc }-\sqrt{\frac{2}{\qd -\qhatc }}\ r(\sin(\theta_\indexl)+\sin(\theta_{N}))\\&+\frac{1}{\qd +\qhatc }-\sqrt{\frac{2}{\qd +\qhatc }}\ r(\cos(\theta_\indexl)+\cos(\theta_{N}))+2r^2\cos(\theta_\indexl)\cos(\theta_{N})\\&-\frac{1}{\qd -\qhatc }+\sqrt{\frac{2}{\qd -\qhatc }}\ r(\sin(\theta_\indexl)+\sin(\theta_{N}))-2r^2\sin(\theta_\indexl)\sin(\theta_{N})\\&=-\frac{1}{\qd +\qhatc }+\frac{1}{\qd -\qhatc }+2r^2(\cos(\theta_\indexl)\cos(\theta_{N})-\sin(\theta_\indexl)\sin(\theta_{N}))\\&=\frac{2\qhatc }{\qd ^2-\qhatc ^2}+2r^2\cos(\theta_\indexl+\theta_{N})\\
&=2\left(\frac{\qhatc }{\qhatd ^2-\qhatc ^2}+r^2\cos(\theta_\indexl+\theta_\indexr)\right).
\end{align*}
Now, we want $\zc=0$ while keeping $r^2$ to a minimum (in order to in turn minimize $\zlr$). We achieve this by putting
\begin{align*}
%r=\sqrt{\frac{|\qhatc |}{\qhatd ^2-\qhatc ^2}},&&\cos(\theta_\indexl+\theta_\indexr)=-\text{sgn}(\qhatc ).
r^2=\frac{|\qhatc |}{\qhatd ^2-\qhatc ^2},&&\cos(\theta_\indexl+\theta_\indexr)=-\text{sgn}(\qhatc ).
\end{align*}
It can be shown that $\qhatd ^2-\qhatc ^2\geq0$ (by inserting \eqref{Apickup} into \eqref{qsnokdefALT} and using that $A^\trans=A\geq0$), therefore the absolute value is only needed for $\qhatc$.
Inserting the above choice of $r^2$ %and $\theta_\indexlr$
 into $\zlr$ in \eqref{zlrsamma} and thereafter using \eqref{qhat} with $\qhatlr=\qhatd$, we obtain
\begin{align*}
\zlr
=\frac{|\qhatc |-\qhatd }{\qhatd ^2-\qhatc ^2}
&=\frac{-1}{\qhatd +|\qhatc |}=-\frac{1}{\qhattot}.
\end{align*}
We have thereby shown that, with $\zc=0$ and $\zl=\zr$ in \eqref{insattansatz}, $1/\qhattot$ is the maximum amount of "positivity" in form of $( \dsel\dsel^\trans+ \dser\dser^\trans)$ we can extract from $A$. Inserting $\zc=0$ and $\zlr=-1/\qhattot$ into
\eqref{insattansatz} and noting that $\widetilde{\scv}^\trans A \widetilde{\scv}\geq0$, we get
\begin{align}\label{happqhat}
%\widetilde{\scv}^\trans A \widetilde{\scv}=\scv^\trans A\scv-\frac{1}{\qhat}\scv^\trans( \dsel\dsel^\trans + \dser\dser^\trans )\scv 
\scv^\trans A\scv-\frac{1}{\qhattot}\scv^\trans( \dsel\dsel^\trans + \dser\dser^\trans )\scv \geq0.
\end{align}
Comparing with \eqref{BorrowProp}, we deduce that %indeed 
$h\app=1/\qhattot$.

\subsection{Proof of Lemma~\ref{happ=1/q}}%that $h\app=1/\q$}
\label{happq=1}

We define $\helpvar=S\scv+M^{-1}\el\arbscall +M^{-1}\er\arbscalr$ and use the relations in \eqref{decompAS} to compute
\begin{align}\label{wTMw}
%\helpvar^\trans M\helpvar&=\scv^\trans A\scv+2\scv^\trans \dsel\arbscall +2\scv^\trans \dser\arbscalr +\q_0\arbscall^2 +2\qc \arbscall \arbscalr+\q_\ngpts\arbscalr^2\\ 
\helpvar^\trans M\helpvar&=\scv^\trans A\scv+2\arbscall\scv^\trans \dsel +2\arbscalr\scv^\trans \dser +\arbscall^2\ql +2 \arbscall \arbscalr\qc+\arbscalr^2\qr 
\end{align}
where $\qall$ are defined in \eqref{qdefcompact} and where $\arbscallr$ are any scalars. 
It is assumed that $M>0$ and that $\helpvar^\trans M\helpvar\geq0$.
Note that the right-hand-side of \eqref{wTMw} has the
same form as \eqref{helpsnok}, but with $\qhatall$ replaced by $\qall$. Thus, by following the same procedure, 
we obtain the relation corresponding to \eqref{happqhat}, namely
\begin{align*}
\scv^\trans A\scv&-\frac{1}{\qtot}\scv^\trans\left( \dsel\dsel^\trans+ \dser\dser^\trans\right) \scv \geq0
\end{align*}
with $\qtot$ defined in \eqref{qdefcompact}.
Comparing with \eqref{BorrowProp} we see that $h\app=1/\qtot$.

%\subsection{Part II: Proof of the relation \eqref{qsnokdefthm}}
\subsection{Proof of Lemma~\ref{qsnoksame}}%Part II: Proof of that $\qsnok_{0,N}$ and $\qsnokc$ in \eqref{qsnokdef} and \eqref{qsnokdefthm} are identical}
%\subsection{Part II: The relation between $\qsnok$ and $\q$}%{Proof of Corollary~\ref{CorGenInv}}
%\begin{proof}[{Proof of Corollary~\ref{CorGenInv}}]
\label{partII}

In \cite{ErikssonDual}, it was shown that $\qsnoklr$ and $\qsnokc$ in \eqref{qsnokdef} can be computed
as
\begin{align}\label{eqqsnok}
%\qsnokl=\el^\trans SK_0S^\trans\el,&&\qsnok_N=\er^\trans SK_0S^\trans\er,&&\qsnokc=\el^\trans SK_0S^\trans\er=\er^\trans SK_0S^\trans\el,
\qsnokl=\dsel^\trans K_0\dsel,&&\qsnokr=\dser^\trans K_0\dser,&&\qsnokc=\dsel^\trans K_0\dser=\dser^\trans K_0\dsel,
\end{align}
%with $\Msnok$ and $K_0$ defined as in \cite{ErikssonDual}.
%with $K_0$ defined (using our notation from \eqref{ASparts}) as 
%\begin{align*}
%K_0=\left[\begin{array}{cc}0&\begin{array}{cc}0\hspace{10pt}&0\hspace{10pt}\end{array}\\\begin{array}{c}0\\0\end{array}&\left[\begin{array}{cc}\Amid&\Aeast\\\Aeast^\trans&\Aon\end{array}\right]^{-1}\end{array}\right].\end{align*}
with $K_0$ defined (using our notation from \eqref{Aparts}) as 
\begin{align*}
K_0=\left[\begin{array}{cc}0&\begin{array}{cc}0\hspace{10pt}&0\hspace{10pt}\end{array}\\\begin{array}{c}0\\0\end{array}&\left[\begin{array}{cc}\Amid&\Aeast\\\Aeast^\trans&\Aend\end{array}\right]^{-1}\end{array}\right].\end{align*}
Now, we want to show that the quantities in \eqref{eqqsnok} are equal to the ones in \eqref{qsnokdefALT}. %\eqref{qsnokdefthm}.
%
%Applying the formula for inverses of block matrices to the above definition of $K_0$, and thereafter using the relation for $\Aon$ in \eqref{Bonus}, we obtain\begin{align}\label{K0}K_0&=\frac{1}{\Aon-\Aeast^\trans\Amid^{-1}\Aeast}\left[\begin{array}{ccc}0&0&0\\0&(\Aon-\Aeast^\trans\Amid^{-1}\Aeast)\Amid^{-1}+\Amid^{-1}\Aeast\Aeast^\trans\Amid^{-1}&-\Amid^{-1}\Aeast\\0&-\Aeast^\trans\Amid^{-1}&1\end{array}\right]\notag\\& =\left[\begin{array}{ccc}0&0&0\\0&\Amid^{-1}&0\\0&0&0\end{array}\right]+\left[\begin{array}{c}0\\-\Amid^{-1}\Aeast\\1\end{array}\right]\left[\begin{array}{ccc}0&-\Aeast^\trans\Amid^{-1}&1\end{array}\right].\end{align}
Applying the formula for inverses of block matrices to the above definition of $K_0$, and thereafter using the relation for $\Aend$ in \eqref{Bonus}, we obtain\begin{align}\label{K0}K_0&=\frac{1}{\Aend-\Aeast^\trans\Amid^{-1}\Aeast}\left[\begin{array}{ccc}0&0&0\\0&(\Aend-\Aeast^\trans\Amid^{-1}\Aeast)\Amid^{-1}+\Amid^{-1}\Aeast\Aeast^\trans\Amid^{-1}&-\Amid^{-1}\Aeast\\0&-\Aeast^\trans\Amid^{-1}&1\end{array}\right]\notag\\& =\left[\begin{array}{ccc}0&0&0\\0&\Amid^{-1}&0\\0&0&0\end{array}\right]+\domsize\left[\begin{array}{c}0\\-\Amid^{-1}\Aeast\\1\end{array}\right]\left[\begin{array}{ccc}0&-\Aeast^\trans\Amid^{-1}&1\end{array}\right].\end{align}
Comparing \eqref{K0} with \eqref{InverseParts} and
\eqref{xuttrycktaiA},
we note that $K_0=\GreenDisc+\xfet\xfet^\trans/\domsize$. Inserting this %expression of $K_0$
 into \eqref{eqqsnok}, and thereafter using \eqref{Grels} and that $\dselr^\trans\ett=0$ and $\dselr^\trans\xfet=1$, yields
\begin{align*}
\qsnokl= -\pickupl^\trans\dsel,&&
\qsnokr=\pickupr^\trans\dser,&&\qsnokc= -\pickupl^\trans\dser= \pickupr^\trans\dsel,
\end{align*}
that is exactly the same relations as in \eqref{qsnokdefALT}.

\section{Explicit inverses of the second derivative operator}
\label{InversesExamples}

We provide the explicit expressions of $\Amid^{-1}$, $\pickuplr$, $\qhatlr$ and $\qhatc$
 for the (2,0), (2,1) and (4,2) order accurate narrow-stencil operators and the (2,0) order accurate wide-stencil operator.
By the notation "(2,0) order accurate operator", we refer to a matrix $D_2$ which has order 2 in the interior finite difference stencil and order 0 at the boundaries.

\subsection{The narrow-stencil (2,0) order operator}
\label{Inverse20}

The simplest possible example of a second derivative operator $D_2$ fulfilling the SBP-properties in \eqref{SBPprop2} is
the 
narrow-stencil (2,0)
order operator, and its corresponding matrix $\OPsnok$ was inverted
already in \cite{Eriksson20092659} for the special case $\alr=1$, $\blr=0$ and $\sigmalr=0$. %$\sigmal=\sigmar=0$.
It is given below, together with its associated %$S$ matrix
$\dselr$ vectors.
\begin{align}\label{D2S20}
D_2=\frac{1}{h^2}\left[\begin{array}{cccccc}0&0\\1&-2&1\\&1&-2&1\\&&\ddots&\ddots&\ddots\\&&&1&-2&1\\&&&&0&0\end{array}\right],
&&
%S=\frac{1}{h}\left[\begin{array}{cccccc}-1&1\\\times&\times&\times&\times&\times&\times\\\times&\times&\times&\times&\times&\times\\
%\vdots&\vdots&\vdots&\vdots&\vdots&\vdots\\\times&\times&\times&\times&\times&\times\\&&&&-1&1\end{array}\right].
\dsel=\frac{1}{h}\left[\begin{array}{c}-1\\1\\0\\\vdots\\0\\0\end{array}\right].
&&
\dser=\frac{1}{h}\left[\begin{array}{c}0\\0\\\vdots\\0\\-1\\1\end{array}\right].
\end{align}
%The interior rows of $S$ are marked with $\times$ because they are not uniquely defined.
The operator $D_2$ is also associated with 
$\PH =h\ \text{diag}\left(\frac{1}{2}, 1, 1, \hdots, 1, 1, \frac{1}{2}\right)$,
%Inserting the above matrices $\PH$, $D_2$ and $S$ into \eqref{Kop}, we obtain the 
and using \eqref{SBPprop2} we obtain the $(\ngpts+1)\times(\ngpts+1)$ matrix $A$ given below. 
%Identifying the parts according to \eqref{ASparts}, gives us the $(\ngpts-1)\times(\ngpts-1)$ matrix $\Amid$. 
The $(\ngpts-1)\times(\ngpts-1)$ matrix $\Amid$ is identified using \eqref{Aparts}.
Gauss--Jordan
elimination then leads to $\Amid^{-1}$ as
\begin{align*}
A\hspace{-1pt}
=\hspace{-1pt}\frac{1}{h}\hspace{-2pt}\left[\hspace{-3pt}\begin{array}{cccccc}1&-1\\-1&2&-1\\&-1&2&-1\\&&\ddots&\ddots&\ddots\\&&&-1&2&-1\\&&&&-1&1\end{array}\hspace{-3pt}\right],&&
%\Amid=\frac{1}{h}\left[\begin{array}{ccccc}2&-1\\-1&2&-1\\&\ddots&\ddots&\ddots\\&&-1&2&-1\\&&&-1&2\end{array}\right]
\Amid^{-1}%=h\left[\begin{array}{cccccc}1-h&1-2h&1-3h&\cdots&2h&h\\1-2h&2(1-2h)&2(1-3h)&\cdots&4h&2h\\1-3h&2(1-3h)&3(1-3h)&\cdots&6h&3h\\\vdots&\vdots&\vdots&\ddots&\vdots&\vdots\\2h&4h&6h&\cdots&2(1-2h)&1-2h\\h&2h&3h&\cdots&1-2h&1-h\end{array}\right].
\hspace{-1pt}=\hspace{-1pt}
h\hspace{-2pt}\left[\hspace{-3pt}\begin{array}{cccc}
 1-\frac{1}{\ngpts}& 1-\frac{2}{\ngpts}&\cdots&\frac{1}{\ngpts}\\
 1-\frac{2}{\ngpts}&2( 1-\frac{2}{\ngpts})&\cdots&\frac{2}{\ngpts}\\
\vdots&\vdots&\ddots&\vdots\\
\frac{1}{\ngpts}&\frac{2}{\ngpts}&\cdots& 1-\frac{1}{\ngpts}
\end{array}\hspace{-3pt}\right]\hspace{-2pt}.
\end{align*}
%\begin{align*}
%\GreenDisc=\left[\begin{array}{ccc}0&0&0\\0&\Amid^{-1}&0\\0&0&0\end{array}\right] 
%%=\left[\begin{array}{cccccc}0&0&0&\cdots&0&0\\0&h(1-h)&h(1-2h)&\cdots&h^2&0\\0&h(1-2h)&2h(1-2h)&\cdots&2h^2&0\\\vdots&\vdots&\vdots&\ddots&\vdots&\vdots\\0&h^2&2h^2&\cdots&h(1-h)&0\\0&0&0&\cdots&0&0\end{array}\right]
%=
%h\left[\begin{array}{cccccc}
%0&0&0&\cdots&0&0\\
%0&1-h&1-2h&\cdots&h&0\\
%0&1-2h&2(1-2h)&\cdots&2h&0\\
%\vdots&\vdots&\vdots&\ddots&\vdots&\vdots\\
%0&h&2h&\cdots&1-h&0\\
%0&0&0&\cdots&0&0
%\end{array}\right].
%\end{align*}
Inserting $\Amid^{-1}$ from above into \eqref{InverseParts}, %and deciphering its pattern,
and using that $x_i=ih$,
yields
\begin{align}\label{Ksnokinv2int}
\left(\GreenDisc\right)_{i,j}=\left\{\begin{array}{ll}
x_j(1-x_i/\domsize),&0\leq j\leq i\leq \ngpts,\\
x_i(1-x_j/\domsize),& 0\leq i \leq j\leq \ngpts.
\end{array}\right.
\end{align}
%where we have used that $x_i=ih$. 
Note the %clear 
striking 
similarity %with
to
 the continuous 
Green's function in
Remark~\ref{RemGreen2}. %, the discrete Green's function from the second order operator is actually point-wise exact.
%Moreover, note that the elements $\left(\GreenDisc\right)_{i,j}\sim\mathcal{O}(1)$ in the "interior" of the matrix, but at most $\mathcal{O}(h)$ for small $j$. 
Next, by noticing the structure of $\dselr$ in \eqref{D2S20} and identifying the first and last columns of $\Amid^{-1}$ as $h(\one-\xvec/\domsize)$ and $h\xvec/\domsize$ we can compute $\GreenDisc\dselr$ and consequently $\pickuplr$ in \eqref{InverseParts} as
 \begin{align*}
 \GreenDisc\dsel=\left[\begin{array}{c}0\\\one-\xvec/\domsize\\0\end{array}\right], &&\GreenDisc\dser=-\left[\begin{array}{c}0\\\xvec/\domsize\\0\end{array}\right],&&
\pickupl=\el,&&\pickupr=\er.
\end{align*} 
Furthermore, inserting these $\pickuplr$ and $\dselr$ from \eqref{D2S20} into
\eqref{qsnokdefALT}, we obtain
 \begin{align*}
\qhatl =\qhatr =1/h,&& \qhatc =0.
 \end{align*}

\subsection{The narrow-stencil (2,1) order operator}
\label{Sec21}

The narrow-stencil (2,1) order operator (see Section C.1 in \cite{Mattsson2004503}), 
%are associated to
have
 the same matrices $\PH$ and $A$ as the (2,0) order operator, and hence its $\GreenDisc$ is % also 
 given by \eqref{Ksnokinv2int}. However, the difference matrices $\dselr$ differ, for the (2,1) order operator they are
\begin{align*}
\dsel^\trans=\frac{1}{h}\left[\begin{array}{ccccccc}-\frac{3}{2}&2&-\frac{1}{2}&0&0&\cdots&0\end{array}\right],&&
\dser^\trans=\frac{1}{h}\left[\begin{array}{ccccccc}0&\cdots&0&0&\frac{1}{2}&-2&\frac{3}{2}\end{array}\right].
\end{align*}
We can compute
$\GreenDisc\dsel$ as
\begin{align*}
\GreenDisc\dsel=h\left[\begin{array}{cccccc}0& 0&0&\cdots&0&0\\0& 1-\frac{1}{\ngpts}& 1-\frac{2}{\ngpts}&\cdots&\frac{1}{\ngpts}&0\\0& 1-\frac{2}{\ngpts}&2( 1-\frac{2}{\ngpts})&\cdots&\frac{2}{\ngpts}&0\\\vdots&\vdots&\vdots&\ddots&\vdots&\vdots\\0&\frac{1}{\ngpts}&\frac{2}{\ngpts}&\cdots& 1-\frac{1}{\ngpts}&0\\0& 0&0&\cdots&0&0\end{array}\right]\frac{1}{h}\left[\begin{array}{c}-\frac{3}{2}\\2\\-\frac{1}{2}\\0\\\vdots\\0\end{array}\right]
%=\left[\begin{array}{cccccc}0& 0&0&\cdots&0&0\\0& 1-\frac{1}{\ngpts}& 1-\frac{2}{\ngpts}&\cdots&\frac{1}{\ngpts}&0\\0& 1-\frac{2}{\ngpts}&2( 1-\frac{2}{\ngpts})&\cdots&\frac{2}{\ngpts}&0\\\vdots&\vdots&\vdots&\ddots&\vdots&\vdots\\0&\frac{1}{\ngpts}&\frac{2}{\ngpts}&\cdots& 1-\frac{1}{\ngpts}&0\\0& 0&0&\cdots&0&0\end{array}\right]\left[\begin{array}{c}-\frac{3}{2}\\2\\-\frac{1}{2}\\0\\\vdots\\0\end{array}\right]
=\left[\begin{array}{c}0\\\frac{3}{2}-\frac{1}{\ngpts}\\1-\frac{2}{\ngpts}\\\vdots\\\frac{1}{\ngpts}\\0\end{array}\right]
%=\ett-\xfet/\domsize+\left[\begin{array}{c}-1\\1/2\\0\\\vdots\\0\\0\end{array}\right]
\end{align*}
and repeating the procedure for $\GreenDisc\dser$ and thereafter using \eqref{InverseParts}, we arrive at
\begin{align*}
\pickupl=\left[\begin{array}{cccccc}1&-\frac{1}{2}&0&\cdots&0&0\end{array}\right]^\trans,&&\pickupr=\left[\begin{array}{cccccc}0&0&\cdots&0&-\frac{1}{2}&1\end{array}\right]^\trans.\end{align*}
Finally, we use \eqref{qsnokdefALT} to compute %or \eqref{qsnokdefthm}, 
\begin{align*}
\qhatlr
=2.5/h,&&\qhatc
=0,
\end{align*}
where $\qhatc
=0$ holds for $\ngpts\geq4$.

\subsection{The narrow-stencil (4,2) order operator}

The operator $D_2$ with fourth order interior accuracy and diagonal norm $\PH$, see Section C.2 in \cite{Mattsson2004503}, 
is associated with the
difference operators
\begin{align}\label{S42}%\begin{split}\dsel^\trans=\frac{1}{h}\left[\begin{array}{cccccccc}-\frac{11}{6}&3&-\frac{3}{2}&\frac{1}{3}&0&0&\cdots&0\\\end{array}\right],\\\dser^\trans=\frac{1}{h}\left[\begin{array}{cccccccc}0&\cdots&0&0&-\frac{1}{3}&\frac{3}{2}&-3&\frac{11}{6}\end{array}\right].\end{split}
\dsel^\trans=\frac{1}{h}\left[\hspace{-2pt}\begin{array}{ccccccc}\frac{-11}{6}&3&\frac{-3}{2}&\frac{1}{3}&0&\cdots&0\\\end{array}\hspace{-2pt}\right],&&\dser^\trans=\frac{1}{h}\left[\hspace{-2pt}\begin{array}{ccccccc}0&\cdots&0&\frac{-1}{3}&\frac{3}{2}&-3&\frac{11}{6}\end{array}\hspace{-2pt}\right].
\end{align}
Using %the first relation in \eqref{Kop}, 
\eqref{SBPprop2} %and \eqref{Kop2}, 
and identifying the interior of $A$ according to \eqref{Aparts}, we obtain 
%$\Amid$. 
\begin{align*}\Amid=\frac{1}{h}\left[\begin{array}{ccccccccc} \frac{59}{24}& -\frac{59}{48}&0& & & && \\ -\frac{59}{48}& \frac{55}{24}& -\frac{59}{48}& \frac{1}{12}& & &&& \\0& -\frac{59}{48}& \frac{59}{24}& -\frac{4}{3}& \frac{1}{12}& &&& \\& \frac{1}{12}& -\frac{4}{3}& \frac{5}{2}& -\frac{4}{3}& \frac{1}{12}&&& \\&& \ddots& \ddots& \ddots& \ddots& \ddots&& \\&&& \frac{1}{12}& -\frac{4}{3}& \frac{5}{2}& -\frac{4}{3}& \frac{1}{12}& \\&&& & \frac{1}{12}& -\frac{4}{3}& \frac{59}{24}& -\frac{59}{48}&0\\&&& & & \frac{1}{12}& -\frac{59}{48}& \frac{55}{24}& -\frac{59}{48}\\&&& & & &0& -\frac{59}{48}& \frac{59}{24}\end{array}\right].\end{align*}
We are now looking for a matrix $\Gmid$ such that
 $\Gmid=\Amid^{-1}$, and %let $\Gmid$ be composed as
 make the ansatz
\begin{align*}\Gmid
 &=\left[\begin{array}{ccccc} \Opinv_{1}&\Opinv_{2}&\hdots&\Opinv_{\ngpts-1} \end{array}\right]
 ,&& \Opinv_{j}=\left[\begin{array}{ccccc} \opinv_{1,j}&\opinv_{2,j}&\hdots&\opinv_{\ngpts-1,j} \end{array}\right]^\trans.
\end{align*}
%We compute $\Amid\Opinv_j$ for $j=1,2,\hdots,\ngpts-1$ as\begin{align}\label{Ksnokjmid}\Amid\Opinv_j &=\frac{1}{48h} \scalebox{.97}{$\left[\begin{array}{c} 118\opinv_{1,j} -59\opinv_{2,j} \\ -59\opinv_{1,j}+ 110\opinv_{2,j} -59\opinv_{3,j}+ 4\opinv_{4,j} \\ -59\opinv_{2,j}+ 118\opinv_{3,j} -64\opinv_{4,j}+4\opinv_{5,j} \\\hline 4\left(\opinv_{2,j} - 16\opinv_{3,j}+ 30\opinv_{4,j} -16\opinv_{5,j}+ \opinv_{6,j} \right) \\\vdots \\4\left(\opinv_{\ngpts-6,j}-16\opinv_{\ngpts-5,j}+ 30\opinv_{\ngpts-4,j} -16\opinv_{\ngpts-3,j}+ \opinv_{\ngpts-2,j} \right) \\\hline4\opinv_{\ngpts-5,j} -64\opinv_{\ngpts-4,j}+ 118\opinv_{\ngpts-3,j} -59\opinv_{\ngpts-2,j}\\4\opinv_{\ngpts-4,j} -59\opinv_{\ngpts-3,j}+ 110\opinv_{\ngpts-2,j} -59\opinv_{\ngpts-1,j}\\ -59\opinv_{\ngpts-2,j}+ 118\opinv_{\ngpts-1,j}\end{array}\right].$}\end{align}
For $\Amid\Gmid=\Imid$ to hold, 
$\Amid\Opinv_j=\evec_j$ must be fulfilled for all $j=1,2,\hdots,\ngpts-1$, 
where the %$(\ngpts-1)\times1$ 
vector $\evec_j=[0\ \hdots\ 0\ 1\ 0\ \hdots\ 0]^\trans$ is non-zero only in its $j$th element.
%From the mid rows of \eqref{Ksnokjmid}, it is clear that we need
From the mid rows of $\Amid\Opinv_j$, given the inner structure of $\Amid$, we thus need 
\begin{align*}
\opinv_{i-2,j} -16\opinv_{i-1,j}+ 30\opinv_{i,j} -16\opinv_{i+1,j}+\opinv_{i+2,j}=12h\delta_{i,j},&&\begin{array}{l}\forall i=4,5,\hdots, \ngpts-4,\\\forall j=1,2,\hdots,\ngpts-1,\end{array} 
\end{align*}
where $\delta_{i,j}$ is the Kronecker delta. 
Hence, the fourth order linear homogeneous recurrence relation $\opinv_{i-2,j} -16\opinv_{i-1,j}+ 30\opinv_{i,j} -16\opinv_{i+1,j}+\opinv_{i+2,j}=0$ has to be fulfilled by almost all $\opinv_{i,j}$. %Using the ansatz $\opinv_{i,j}=\rec^i$, we obtain\begin{align*}1 -16\rec+ 30\rec^{2} -16\rec^{3}+ \rec^{4}=0,\end{align*}which has the roots $\rec_{1,2}=1$ and $\rec_{3,4}=7\pm\sqrt{48}$ (note that $\rec_4=\rec_3^{-1}$ and that $\rec_{1,2}=1$ is a double root).
%Therefore, t
The explicit solution to this recursive relation has the form 
$\opinv_{i,j}=\cc{j}+\ck{j}i+\ct{j}\theR^i+\cf{j}\theR^{-i}$,
where $\theR=7+\sqrt{48}\approx 13.9$ and where $\call$ are $j$-dependent constants.
To be precise, $\opinv_{i,j}$ has this form for $2\leq i\leq \ngpts-2$, and
 we need two versions of the $j$-dependent constants, that is $\opinv_{i,j}=\ucc{j}+\uck{j} i+\uct{j}\theR^i+\ucf{j}\theR^{-i}$ for $2\leq i\leq j$ and $\opinv_{i,j}=\lcc{j}+\lck{j} i+\lct{j}\theR^i+\lcf{j}\theR^{-i}$ for $j\leq i\leq \ngpts-2$. 
For each $j=2,3,\hdots,\ngpts-2$, we thus have eight unknown constants $\callu$ and $\calll$, as well as the two remaining unknowns $ \opinv_{1,j} $ and $ \opinv_{\ngpts-1,j}$. These are determined by the
 three first and the three last rows in 
the requirement $\Amid\Opinv_j=\evec_j$, which %\eqref{innercolumns} 
gives us six conditions. From the rows $i=j-1,j,j+1$, we get three more conditions and in addition, we demand that the two versions of $\opinv_{j,j}$ are identical. Altogether, this leads to a $10\times10$ system of equations which we solve using Gauss--Jordan
elimination. 
 The boundary columns $j=1$ and $j=\ngpts-1$ must be treated separately, in a similar manner. %Together, this leads to
 All in all, these steps lead to the elements of the inverse $(\Amid^{-1})_{i,j}=\opinv_{i,j}$ as
 \begin{align*}
 (\Amid^{-1})_{i,j}%=\opinv_{i,j}
 &=\corrgreen_{i,j}+\left\{\begin{array}{ll}
 x_j(1-x_ i/\domsize),&1\leq j\leq i\leq \ngpts-1\\x_i(1- x_j/\domsize ),&1\leq i\leq j\leq \ngpts-1,\end{array}\right.
 \end{align*}
% where we have used that $x_i=ih$ and that $h=1/\ngpts$. 
which is thus similar to the second order version of $\Amid^{-1}$, % (and thus also to $\GreenDisc$, compare Remark~\ref{RemGreen}), 
plus
an additional term 
%$\corrgreen_{i,j}\sim\mathcal{O}(h)$.
$\corrgreen_{i,j}$.
 This additional correction term is, for $2\leq i,j\leq \ngpts-2$, given by
 \begin{align*}
 \corrgreen_{i,j}&=\left\{\begin{array}{ll} -h\frac{\scalebox{1}{$\nyserie_{j}\nyserie_{\ngpts-i} $}}{\scalebox{1}{$\nydetskalning_\ngpts$}},&2\leq j\leq i\leq \ngpts-2,\\\vspace{-8pt}\\
 -h\frac{\scalebox{1}{$\nyserie_{i}\nyserie_{\ngpts-j}$}}{\scalebox{1}{$\nydetskalning_\ngpts$}},&2\leq i\leq j\leq \ngpts-2,\end{array}\right.
 \end{align*}
 where 
\begin{align*}
%\nyserie_i=\frac{(51-2\theR^{-1})\theR^{i-2}-(51-2\theR)\theR^{2-i}}{8\sqrt{3}},&&\nydetskalning_\ngpts=\frac{\theR^{\ngpts-4}(2\theR^{-1}-51 )^2-\theR^{4-\ngpts}(2\theR-51)^2}{8\sqrt{3}}.
\nyserie_i=\frac{(51-2\theR^{-1})\theR^{i-2}-(51-2\theR)\theR^{2-i}}{\theR-\theR^{-1}},&&\nydetskalning_\ngpts=\frac{\theR^{\ngpts-4}(2\theR^{-1}-51 )^2-\theR^{4-\ngpts}(2\theR-51)^2}{\theR-\theR^{-1}}.
\end{align*}%$\nyserie_i=\frac{(51-2\theR^{-1})\theR^{i-2}-(51-2\theR)\theR^{2-i}}{8\sqrt{3}}$.
Note that %$\deta\detd-\detb\detc=\detskalning\pendel$ and that 
$\nydetskalning_\ngpts\neq0$ (unless $\ngpts\approx 3.7$), so there is no risk of division by zero. 
Moreover, for $i,j=1$ or $i,j=\ngpts-1$ we have%the first and last rows/columns are
 \begin{align*}
\corrgreen_{1,j} 
 &=-h \frac{\nyserie_{\ngpts-j} }{\nydetskalning_\ngpts},&
\corrgreen_{\ngpts-1,j}
 &=-h \frac{\nyserie_{j} }{\nydetskalning_\ngpts},&&2\leq j\leq \ngpts-2, 
\\% \end{align*}
%Correspondingly, the first and last column, for the inner rows %($ 2\leq i\leq \ngpts-2$)
% we have
%\begin{align*}
\corrgreen_{i,1}&=-h\frac{\nyserie_{\ngpts-i}}{\nydetskalning_\ngpts},&
\corrgreen_{i,\ngpts-1}&=-h\frac{\nyserie_{i}}{\nydetskalning_\ngpts},&&2\leq i\leq \ngpts-2,
\end{align*}
and %finally, the corners elements are%$\omega_0=2(8\sqrt{3})$
\begin{align*}
\corrgreen_{1,1}=\corrgreen_{\ngpts-1,\ngpts-1}&=-h\frac{\nyserie_{\ngpts-2} }{2\nydetskalning_\ngpts}-h \frac{11}{118},&
\corrgreen_{1,\ngpts-1}=\corrgreen_{\ngpts-1,1}&=-h\frac{\nyserie_2}{2\nydetskalning_\ngpts}.
\end{align*}
From \eqref{InverseParts} we have that the interior of $\GreenDisc$ is given by $\Amid^{-1}$ described above.
Next, we %identify $\Swest=\frac{1}{h}\left[\begin{array}{cccccc}3&-\frac{3}{2}&\frac{1}{3}&0&\hdots&0\end{array}\right]^\trans$ and $\Seast=\frac{1}{h}\left[\begin{array}{cccccc}0&\hdots&0&\frac{1}{3}&-\frac{3}{2}&3\end{array}\right]^\trans$
use $\dsel$ from \eqref{S42} to compute $\GreenDisc\dsel$ and thereafter 
%From \eqref{yochz} we can now compute
 \eqref{InverseParts} again, % either or \eqref{qsnokdefthm}, 
to compute
%\begin{align*}
%\Amid^{-1}\Swest=\one-\xvec-
%\left[\begin{array}{c}
%-\frac{5}{6}+ 3\frac{11}{118} +\frac{1}{3}\frac{\nyserie_{N-3} }{\nydetskalning_N}\\\frac{1}{3}+\frac{2}{3}\frac{\nyserie_{N-3}}{\nydetskalning_N}\\
%17\frac{\nyserie_{N-3} }{\nydetskalning_N}\\\vdots\\
%17\frac{\nyserie_{2} }{\nydetskalning_N}\\
% \frac{17 }{\nydetskalning_N}\end{array}\right]
% %&&\Amid^{-1}\Seast=
%\end{align*}
%leading to \eqref{qsnokdefthm}
%the narrow-stencil (4,2) order version of 
$\pickupl$ as
\begin{align}\label{pickup42}
(\pickupl)_i=\left\{\begin{array}{cl}1&i=0\\
%-\frac{98}{177}+\frac{1}{3}\frac{\nyserie_{N-3}}{\nydetskalning_N}&i=1\\\frac{1}{3}+\frac{2}{3}\frac{\nyserie_{N-3}}{\nydetskalning_N}&i=2\\17\frac{\nyserie_{N-i} }{\nydetskalning_N}&i=3,4,\hdots,N-2,\\
-\frac{85}{118}+\frac{17}{2}\frac{\nyserie_{\ngpts-2}}{\nydetskalning_\ngpts}&i=1\\17\frac{\nyserie_{\ngpts-i} }{\nydetskalning_\ngpts}&i=2,3,\hdots,\ngpts-2,\\
 \frac{17 }{\nydetskalning_\ngpts}&i=\ngpts-1 \\0&i=\ngpts\end{array}\right.
 &&
 \lim_{\ngpts\to\infty}\pickupl=%\approx
 \left[\begin{array}{c}1\\
-0.5532\hdots\hspace{8pt}\\
0.3342\hdots\\
0.0239\hdots\\
%0.0017\hdots\\
\vdots\\
0\end{array}\hspace{-3pt}\right],
 \end{align}
where we have used that $\nydetskalning_\ngpts+2\nyserie_{\ngpts-3}=51\nyserie_{\ngpts-2}$. %or $\frac{1}{3}\frac{\nyserie_{N-3}}{\nydetskalning_N}=\frac{17}{2}\frac{\nyserie_{N-2}}{\nydetskalning_N}-\frac{1}{6}$ FIXA NEDAN OCKS{\AA}
Then, $\pickupr$ is given by $(\pickupr)_i=(\pickupl)_{\ngpts-i}$.
 We also compute
 the scalars from \eqref{qsnokdefALT}, as
\begin{align*}
%\qsnokl=\qsnok_N=\frac{1}{h}\left(\frac{707}{177}-\frac{17}{3}\frac{\nyserie_{N-3} }{\nydetskalning_N}\right),
%&&\qsnokc=\frac{1}{h} \frac{ 17^2}{\nydetskalning_N}.
\qhatl=\qhatr=\frac{1}{h}\left(\frac{2417}{354}-\frac{17^2\nyserie_{\ngpts-2}}{2\nydetskalning_\ngpts}\right),
&&\qhatc=\frac{1}{h} \frac{ 17^2}{\nydetskalning_\ngpts}.
\end{align*}
Evaluating %the quantities 
$h\qhatlr$ and $h\qhatc$ explicitly for some values of $\ngpts$, see Table~\ref{QRtab1}, we
 see that these numbers corresponds exactly (to machine precision) to $\qsnokl h$ and $\qsnokc h$ tabulated in \cite{ErikssonDual}. % computed as stated in \eqref{eqqsnok}. 
This 
serves as a numerical verification of %Corollary~\ref{CorGenInv},
Lemma~\ref{qsnoksame} and indirectly of
Theorem~\ref{ThmGenInv}.

\begin{table}[h]\centering$
\begin{array}{|r|ll|}\hline
\ngpts&h\qhatlr&h\qhatc\\\hline
%1& 0.584052874869253& -0.002766953574540\\
%2& 0.583618531796803& -0.038543611629768\\
%3& 0.497207425343018& -0.550476190476190\\
%4& 4.874980913116506 &1.952702702702702\\
%5& 3.989986272731031& 0.111282248748556\\
%6& 3.986369061828599& 0.007981220657277\\
%7& 3.986350435820677& 0.000573022724614\\
8& 3.986350339808304& 0.000041141179445\\
 9& 3.986350339313381& 0.000002953803786\\
% 10& 3.986350339310829& 0.000000212073570\\
 10& 3.986350339310831& 0.000000212073570\\
 11& 3.986350339310817& 0.000000015226197\\
 12& 3.986350339310817& 0.000000001093192\\
% 13& 3.986350339310817& 0.000000000078488\\
% 14& 3.986350339310817& 0.000000000005635\\
% 15& 3.986350339310817& 0.000000000000405\\
% 16& 3.986350339310816& 0.000000000000029\\
% 17& 3.986350339310817& 0.000000000000002\\
% 18& 3.986350339310817& 0.000000000000000\\
% 19& 3.986350339310817& 0.000000000000000\\
% 20& 3.986350339310817& 0.000000000000000\\
\hline
\end{array}$
\caption{ The parameters $h\qhatlr$ and $h\qhatc$ %from \eqref{nyaq} 
 in the (4,2) order case
evaluated explicitly.}
\label{QRtab1}
\end{table}

\subsection{The wide-stencil (2,0) order operator}
\label{AppWide}

The wide-stencil (2,0) order accurate operator $D_2$, which is obtained by squaring the (2,1) order accurate operator $D_1$ from \eqref{D1Qsnok}, is given below together with $\dselr=D_1^\trans\elr$
\begin{align*}
%\scalebox{.93}{$
%D_2\hspace{-1pt}=\hspace{-1pt}\frac{1}{h^2}\hspace{-3pt}
%\left[\hspace{-3pt}\begin{array}{ccccccc}\frac{1}{2}& -1& \frac{1}{2}\\
%\frac{1}{2}& -\frac{3}{4}& 0& \frac{1}{4}& \\
%\frac{1}{4} &0& -\frac{1}{2}& 0& \frac{1}{4}\\
%&\ddots&\ddots&\ddots&\ddots&\ddots\\
%&&\frac{1}{4} &0& -\frac{1}{2}& 0& \frac{1}{4}\\
%&&&\frac{1}{4} &0& -\frac{3}{4}& \frac{1}{2}\\
%&&&&\frac{1}{2} &-1& \frac{1}{2}\end{array}\hspace{-3pt}\right]\hspace{-3pt},\hspace{6pt}
%S\hspace{-1pt}=\hspace{-1pt}\frac{1}{h}\hspace{-3pt}\left[\hspace{-3pt}\begin{array}{ccccccc}-1&1\\-\frac{1}{2}&0&\frac{1}{2}\\&-\frac{1}{2}&0&\frac{1}{2}\\&&\ddots&\ddots&\ddots\\
%&&&-\frac{1}{2}&0&\frac{1}{2}\\&&&&-\frac{1}{2}&0&\frac{1}{2}\\&&&&&-1&1\end{array}\hspace{-3pt}\right]\hspace{-3pt}.%$}
D_2=\frac{1}{h^2}
\left[\begin{array}{ccccccc}\frac{1}{2}& -1& \frac{1}{2}\\
\frac{1}{2}& -\frac{3}{4}& 0& \frac{1}{4}& \\
\frac{1}{4} &0& -\frac{1}{2}& 0& \frac{1}{4}\\
&\ddots&\ddots&\ddots&\ddots&\ddots\\
&&\frac{1}{4} &0& -\frac{1}{2}& 0& \frac{1}{4}\\
&&&\frac{1}{4} &0& -\frac{3}{4}& \frac{1}{2}\\
&&&&\frac{1}{2} &-1& \frac{1}{2}\end{array}\right],
&&
\dsel=\frac{1}{h}\left[\begin{array}{c}-1\\1\\0\\\vdots\\0\\0\\0\end{array}\right],
&&
\dser=\frac{1}{h}\left[\begin{array}{c}0\\0\\0\\\vdots\\0\\-1\\1\end{array}\right].
\end{align*}
The operator is also associated with the same
%$\PH =h\ \text{diag}\left(\left[\begin{array}{ccccccc}\frac{1}{2}&1&1&\cdots&1&1&\frac{1}{2}\end{array}\right]\right)$,
$\PH =h\ \text{diag}\left(\begin{array}{ccccccc}\frac{1}{2},& 1,& 1,& \hdots,& 1,& 1,& \frac{1}{2}\end{array}\right)$ as the other operators with second order accuracy,
%Inserting the above matrices $\PH$, $D_2$ and $S$ into \eqref{Kop}, we obtain the 
and from this we can compute the $(\ngpts+1)\times(\ngpts+1)$ matrix $A$. Identifying the parts of $A$ according to \eqref{Aparts}, gives us the $(\ngpts-1)\times(\ngpts-1)$ matrix $\Amid$. The inverse of this matrix $\Amid$ is
%$D_1^T\PH D_1=A$
\begin{align*}
\Amid^{-1}
=
2h\left[\begin{array}{cccc}
1-\frac{1}{\ngpts}&0&1-\frac{3}{\ngpts}&\cdots\\
0&2(1-\frac{2}{\ngpts})&0&\cdots\\
1-\frac{3}{\ngpts}&0&3(1-\frac{3}{\ngpts})&\cdots\\
\vdots&\vdots&\vdots&\ddots
\end{array}\right],
\end{align*}
that is the discrete Green's function in \eqref{InverseParts} becomes
\begin{align*}
\left(\GreenDisc\right)_{i,j}=\left\{\begin{array}{ll}
x_j(1-x_i/\domsize)(1+(-1)^{i+j}),&0\leq j\leq i\leq \ngpts,\\
x_i(1-x_j/\domsize)(1+(-1)^{i+j}),& 0\leq i \leq j\leq \ngpts.
\end{array}\right.
\end{align*}
Thus the discrete Green's function produced by the wide operator oscillate, jumping between 0 and 2 times the exact value. 
Next, 
%\begin{align*}
%\extra_0=2\left[\begin{array}{c}0\\1-h\\0\\1-3h\\\vdots\\0\end{array}\right],&&\extra_N=2\left[\begin{array}{c}0\\\vdots\\1-3h\\0\\1-h\\0
%\end{array}\right],
%\end{align*}
using \eqref{InverseParts} we obtain the vectors
\begin{align*}
%\pickupl=\left[\begin{array}{c}1\\-(1-h)\\1-2h\\-(1-3h)\\\vdots\\\pm 2h\\\mp h\\0\end{array}\right],&&\pickupr=\left[\begin{array}{c}0\\-(-1)^N h\\(-1)^N 2h\\\vdots\\-(1-3h)\\1-2h\\-(1-h)\\1\end{array}\right],
\pickupl^\trans&=\left[\begin{array}{cccccccccc}1&-(1-\frac{1}{\ngpts})&1-\frac{2}{\ngpts}&-(1-\frac{3}{\ngpts})&\hdots&(-1)^\ngpts\frac{2}{\ngpts}&-(-1)^\ngpts \frac{1}{\ngpts}&0\end{array}\right],\\
\pickupr^\trans&=\left[\begin{array}{cccccccccc}0&-(-1)^\ngpts \frac{1}{\ngpts}&(-1)^\ngpts \frac{2}{\ngpts}&\hdots&-(1-\frac{3}{\ngpts})&1-\frac{2}{\ngpts}&-(1-\frac{1}{\ngpts})&1\end{array}\right].
\end{align*}
Last,
%we identify %$\Swest=\evec_0/h$ and $\Seast=\evec_N/h$ {\kom variables $\vec{e}_0$? Definiera eller ta bort}
%$\Swest=\frac{1}{h}\left[\begin{array}{ccccc}1&0&0&\hdots&0\end{array}\right]^\trans$ and $\Seast=\frac{1}{h}\left[\begin{array}{ccccc}0&\hdots&0&0&1\end{array}\right]^\trans$ in $\dselr$ above,
%according to %the notation in
% \eqref{ASparts}. We compute the (2,0) order wide-stencil version of \eqref{qsnokdefthm}, as
we compute the (2,0) order wide-stencil version of \eqref{qsnokdefALT}, as
\begin{align*}
\qhatl=\qhatr
=\frac{2}{h}-1/\domsize,&
&\qhatc=-(-1)^{\ngpts}/\domsize.
\end{align*}
In the wide-stencil case, $\qlr=\elr^\trans\PH^{-1}\elr=2/h$ and $\qc=\elr^\trans\PH^{-1}\erl=0$ can be computed directly. We recall that $\qsnokall=\qhatall$ and note that %as mentioned in Section~\ref{relationsstabsing}, 
$\qsnoklr\neq\qlr$ and $\qsnokc\neq\qc$, but still $\qsnoktot=\qtot=2/h$. 
%$\qsnokl+|\qsnokc|=2/h$, which is as expected since $\q=\q_0+|\q_c|=\el^\trans\PH^{-1}\el+0=2/h$, compare 
%Appendix~\ref{Appqqsnok}.) %Remark~\ref{RemWideq}.) 
Compare with the discussion in Section~\ref{relationsErikssonDual}.

\bibliographystyle{plain} 
\bibliography{Acc} %

%\newpage
%\tableofcontents 

\end{document}